\documentclass[conference]{IEEEtran}
\usepackage{times}

\usepackage{amsmath,amsfonts,bm,amsthm,mathtools}

\newcommand{\T}{^\mathrm{T}}
\newcommand{\tr}{\mathrm{tr}}

\def\1{\bm{1}}

\newcommand{\Rb}{\mathbb{R}}

\newcommand{\Eb}{\mathbb{E}}

\newcommand{\ExP}[2]{\Eb_{{#1}}{\left[#2\right]}}

\def\vb{{\bm{b}}}

\def\vk{{\bm{k}}}

\def\vo{{\bm{o}}}

\def\vu{{\bm{u}}}

\def\vx{{\bm{x}}}

\newcommand*{\vK}{{\bf K}}

\newcommand*{\vU}{{\bf U}}

\newcommand*{\vX}{{\bf X}}

\DeclareMathAlphabet{\mathsfit}{\encodingdefault}{\sfdefault}{m}{sl}
\SetMathAlphabet{\mathsfit}{bold}{\encodingdefault}{\sfdefault}{bx}{n}

\DeclarePairedDelimiter\abs{\lvert}{\rvert}
\DeclarePairedDelimiter\norm{\lVert}{\rVert}

\newtheorem{theorem}{Theorem}
\newtheorem{assumption}{Assumption}
\newtheorem{lemma}{Lemma}
\newtheorem{definition}{Definition}

\newtheorem{innercustomgeneric}{\customgenericname}
\providecommand{\customgenericname}{}
\newcommand{\newcustomtheorem}[2]{%
  \newenvironment{#1}[1]
  {%
   \renewcommand\customgenericname{#2}%
   \renewcommand\theinnercustomgeneric{##1}%
   \innercustomgeneric
  }
  {\endinnercustomgeneric}
}

\newcustomtheorem{customlemma}{Lemma}

\theoremstyle{remark}

\newcommand*\dif{\mathop{} \mathrm{d}}

\newcommand*\du{\dif u}

\newcommand*\vxbar{\bar{\vx}}
\newcommand*\vubar{\bar{\vu}}

\newcommand*\ui{u^{i}}
\newcommand*\Ui{U^{i}}
\newcommand*\dui{\dif u^{i}}
\newcommand*\vUs{\vU^{*}}

\newcommand*\feasi{\mathcal{U}^i}

\newcommand*\negi{\neg i}
\newcommand*\unegi{u^{\negi}}
\newcommand*\vunegi{{\vu}^{\negi}}
\newcommand*\vUnegi{{\vU}^{\negi}}

\newcommand*\pii{{\pi}^{i}}
\newcommand*\pinegi{{\pi}^{\negi}}

\newcommand*\piis{{\pi}^{i*}}
\newcommand*\pinegis{{\pi}^{\negi*}}

\newcommand*\vdx{\delta \vx}
\newcommand*\vdu{\delta \vu}
\newcommand*\vdus{{\delta \vu^*}}
\newcommand*\vdunegi{\delta {\vu^{\negi}}}
\newcommand*\vdunegis{\delta {\vu^{\negi*}}}

\newcommand*\vtildedu{\delta \tilde{\vu}}

\newcommand*\delu{\delta u}
\newcommand*\delui{\delta {u^i}}
\newcommand*\deluis{\delta {u^{i*}}}
\newcommand*\ddelui{\dif \delta {u^i}}

\newcommand*\Quui{{Q_{\ui\ui}^i}}

\newcommand*\nonego{{\neg 1}}
\newcommand*\ind{\mathbbm{1}}
\usepackage[numbers]{natbib}
\usepackage{multicol}
\usepackage{xr-hyper}
\usepackage[bookmarks=true]{hyperref}
\usepackage{times}

\usepackage[nolist]{acronym}
\begin{acronym}
\acro{DDP}{Differential Dynamic Programming}
\acro{iLQR}{iterative Linear Quadratic Regulator}
\acro{IBR}{Iterative Best Response}
\acro{MMELQGames}{Multimodal Maximum Entropy Linear Quadratic GAMES}
\acro{MELQGames}{Maximum Entropy Linear Quadratic GAMES}
\acro{GNEP}{Generalized Nash Equilibrium Problem}
\acro{GNE}{Generalized Nash Equilibrium}

\acro{POMDP}{Partially Observable Markov Decision Process}

\acro{MaxEnt}{Maximum Entropy}
\acro{MCMC}{Markov chain Monte Carlo}
\acro{VI}{Variational Inference}

\acro{LICQ}{Linear Independent Constraint Qualification}
\acro{SOSC}{Second-order Sufficient Optimality Condition}
\end{acronym}
\usepackage{hyperref}
\hypersetup{colorlinks = true,
            linkcolor = blue,
            urlcolor  = black,
            citecolor = blue,
            anchorcolor = blue}

\usepackage{url}

\usepackage[utf8]{inputenc} 
\usepackage[T1]{fontenc}    
\usepackage{url}            
\usepackage{booktabs}       
\usepackage{amsfonts}       
\usepackage{nicefrac}       
\usepackage{microtype}      
\usepackage{amsmath}
\usepackage{amssymb}
\usepackage{bm}
\usepackage{mathtools}
\usepackage{amsthm}
\usepackage{mathrsfs}
\usepackage{subcaption}
\usepackage{graphicx} 
\usepackage{caption}

\usepackage{algorithm}
\usepackage{algpseudocode}
\usepackage{algorithmicx}

\usepackage[nolist]{acronym}
\usepackage{bbm}
\usepackage{xcolor}
\usepackage{wrapfig}
\usepackage{siunitx}
\usepackage{multirow}

\let\labelindent\relax
\usepackage{enumitem}

\usepackage[symbol]{footmisc}
\usepackage{float}
\usepackage[capitalise]{cleveref}
\usepackage{caption}

\usepackage{placeins}

\usepackage{tikz}
\usepackage{etoolbox}
\usepackage{siunitx}
\usepackage{booktabs}
\usepackage{natbib}

\robustify\bfseries

\algrenewcommand\algorithmicindent{1.0em}%

\algblockdefx{MRepeat}{EndRepeat}{\textbf{repeat}}{}
\algnotext{EndRepeat}

\def\joinedsupplementary{1}

\ifx\joinedsupplementary\undefined
    \makeatletter
    \newcommand*{\addFileDependency}[1]{
      \typeout{(#1)}
      \@addtofilelist{#1}
      \IfFileExists{#1}{}{\typeout{No file #1.}}
    }
    \makeatother
    
    \newcommand*{\myexternaldocument}[1]{%
        \externaldocument{#1}%
        \addFileDependency{#1.tex}%
        \addFileDependency{#1.aux}%
    }
    \myexternaldocument{paper/supplementary}
\fi

\IEEEoverridecommandlockouts
\begin{document}

\title{Multimodal Maximum Entropy Dynamic Games}




%
\author{\authorblockN{Oswin So\authorrefmark{1},
Kyle Stachowicz
and Evangelos A. Theodorou%
\thanks{Toyota Research Institute provided funds to support this work.}
}
\authorblockA{Autonomous Control and Decision Systems Lab\\
Georgia Institute of Technology, Atlanta, Georgia\\
\authorrefmark{1}Correspondence to: \href{mailto:oswinsog@gatech.edu}{oswinso@gatech.edu}}}

\maketitle

\begin{abstract}
Environments with multi-agent interactions often result a rich set of modalities of behavior between agents due to the inherent suboptimality of decision making processes when agents settle for satisfactory decisions.
However, existing algorithms for solving these dynamic games are strictly unimodal and fail to capture the intricate multimodal behaviors of the agents.
In this paper, we propose MMELQGames (Multimodal Maximum-Entropy Linear Quadratic Games), a novel constrained multimodal maximum entropy formulation of the Differential Dynamic Programming algorithm for solving generalized Nash equilibria.
By formulating the problem as a certain dynamic game with incomplete and asymmetric information where agents are uncertain about the cost and dynamics of the game itself,
the proposed method is able to reason about multiple local generalized Nash equilibria,
enforce constraints with the Augmented Lagrangian framework
and also perform Bayesian inference on the latent mode from past observations.
We assess the efficacy of the proposed algorithm on two illustrative examples: multi-agent collision avoidance and autonomous racing.
In particular, we show that only MMELQGames is able to
effectively block a rear vehicle when given a speed disadvantage and the rear vehicle can overtake from multiple positions.

\end{abstract}

\IEEEpeerreviewmaketitle

\section{Introduction}
\label{sec:intro}
Planning in multi-agent scenarios is a challenging task -- actions taken by any robot cannot be considered in isolation and must take the response from other agents into account. Classical approaches to tackling this problem adopt a predict-then-plan architecture, where the actions of other agents are assumed to be independent of the ego-agent's actions.
However, this assumption breaks down in the context of multi-agent games such as racing, surveillance or autonomous driving where agents may have conflicting interests and it is advantageous to manipulate the behavior of other agents.

Recent works have introduced game-theoretic formulations to address this problem
\cite{cleac2019algames, fridovich2020efficient, laine2021multi,  mehr2021maximum, schwarting2021stochastic, wang2019game, wang2020multi}.
By assuming that agents act rationally, the optimal actions of other agents can be predicted by solving for the Nash equilibria of the non-cooperative game assuming that the true objectives of other agents are known.
However, in practice humans face cognitive limitations and seek decisions that are satisfactory but often suboptimal, a phenomenon described by the concept of ``bounded rationality'' \cite{simon1972theories}.

This disconnect is often modeled using the the \ac{MaxEnt} framework which models actions of agents as being stochastic in nature.
\ac{MaxEnt} has been successfully applied to diverse areas including inverse reinforcement learning \cite{mehr2021maximum, ziebart2008maximum}, forecasting \cite{evans2021maximum}, and biology \cite{de2018introduction}.
In particular, \citet{mehr2021maximum} propose a game-theoretic maximum entropy algorithm for finding nash equilibria policies to dynamic games via an extension of the iLQGames method \cite{fridovich2020efficient}.

\begin{figure}[!t]
    \centering
    \parbox[b]{.49\columnwidth}{\centering Overtake Above}
    \parbox[b]{.49\columnwidth}{\centering Overtake Below}\\
    \vspace{.25em}
    \includegraphics[width=0.99\columnwidth ]{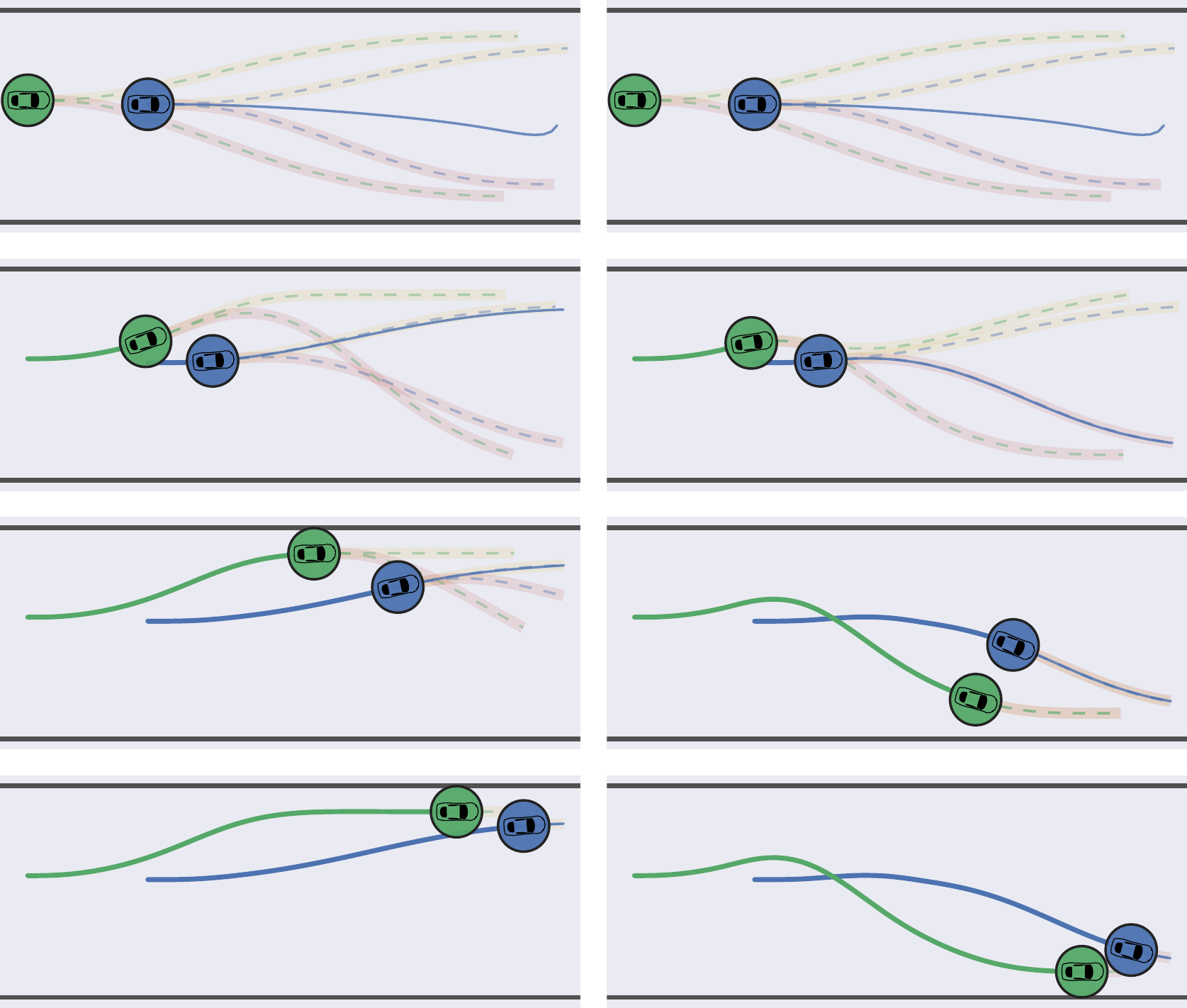}
    \caption{
    Two-agent autonomous racing scenario where the rear agent (green) has a higher maximum velocity than the lead agent (blue).
    Different highlighted colors representing different planned modes.
    The planned trajectory for the ego agent is shown as a thin blue line.
    The lead agent is able to reason about and correctly infer the different modes that the rear agent could take.
    Furthermore, when the ego agent is uncertain about which mode the rear agent will choose, it hedges against both possibilities by planning a path down the middle.
    }
    \label{fig:eyecandy}
\end{figure}
However, these methods consider only a single local minimum. As a result, the resulting policies are \textit{unimodal} and fail to capture the \textit{multimodal} nature of solutions under the \ac{MaxEnt} framework.
In the presence of multiple local minima, the true maximum-entropy policy will be multimodal, with one mode corresponding to each local minimum.
However, computing expectations over multimodal policies in multi-agent settings quickly leads to combinatorial explosion and is computationally intractable
when solving for multimodal Nash equilibria.

By considering a novel \ac{MaxEnt} dynamic game with incomplete information and information asymmetry,
we convert the previous challenge into that of solving a POMDP for the ego agent where the discrete latent variable corresponds to different local generalized \ac{MaxEnt} Nash equilibria found by running a constrained version of \ac{MELQGames} in parallel.
This approach further allows for Bayesian inference over the discrete latet variable by considering past observations of controls from non-ego agents.

The main contributions of our work are threefold:
\begin{itemize}
    \item We propose the constrained \ac{MaxEnt} dynamic game setting to handle bounded rationality in decison making under inequality constraints and extend \ac{MELQGames} to solve generalized Nash equilibrium.
    \item We consider a novel \ac{MaxEnt} dynamic game with incomplete and asymmetric information. We provide a a computationally efficient solution that allows for Bayesian inference of the latent mode when the underlying control policies are multimodal in nature. 
    \item We showcase the benefits of the proposed algorithm in simulation, including an autonomous racing example. The results demonstrate the superiority of \ac{MMELQGames} against other game theoretic formulations in terms of successfully predicting agents behavior in dynamic settings. 
\end{itemize}

\section{Problem Setup}
\label{sec:problem_setup}
In this section, we introduce the constrained \ac{MaxEnt} dynamic game and the corresponding \ac{GNEP}.

\subsection{Discrete Dynamic Games and Generalized Nash Equilibria}
We consider a discrete dynamic game with $N$ players with joint controls
$\vu_t = [u_t^1, \dots, u_t^N] = [u^i_t, \vunegi_t]\in \Rb^{n_\vu}$,
where $u^i_t$ denotes the control input of player $i$ and
we use $\vunegi_t$ to denote the controls excluding the $i$th agent.
We denote $\vx_t \in \Rb^{n_\vx}$ the joint state of the system at timestep $t$ which evolves under the discrete-time dynamics
\begin{equation} \label{eq:pf:dyn}
    \vx_{t+1} = f(\vx_t, u^1_t, \dots, u^N_t) = f(\vx_t, \vu_t).
\end{equation}
Each agent's objective is to minimize a corresponding cost function $J^i$ in finite-horizon $T$
with running cost $l^i$ and terminal cost $\Phi^i$, where $\vU = [\vu_1, \dots, \vu_{T-1}]$ denotes a
trajectory
\begin{equation} \label{eq:pf:obj_fn}
    J^i(\vU) = \Phi^i(\vx_T) + \sum_{t=1}^{T-1} l^i(\vx_t, \vu_t).
\end{equation}
The goal of agent $i$ is to choose controls $\Ui$ that optimize $J^i$ while respecting inequality constraints $h^i$:
\begin{equation} \label{eq:pf:opt}
\begin{aligned}
    \min_{u^i} \quad J^i(\vU),
    \qquad \textrm{s.t.} \quad  h^i(\vU) \leq 0.
\end{aligned}
\end{equation}
Note that both the objective function and the constraints are a function of $U^i$ and $\vUnegi$.
Due to the presence of constraints $h^i$ which may couple agent $i$'s feasible control set to other agent's controls $\vUnegi$, \eqref{eq:pf:opt} is a \ac{GNEP} \cite{facchinei2010generalized}.
Let
\begin{equation}
  \feasi(\vUnegi) = \{ \Ui | h^i(\vU) \leq 0 \},
\end{equation}
denote the feasible control set for agent $i$ given the controls of other agents $\vUnegi$.
Solving an (open loop) \ac{GNEP} amounts to finding controls $\vUs$ such that
for each agent $i$,
\begin{equation} \label{eq:pf:gnep}
    J^i(U^{i*}, \vU^{\negi*}) \leq J^i(\Ui, \vU^{\negi*}),
    \quad \forall \,\Ui \in \feasi(\vUnegi).
\end{equation}
However, since solving for the global \ac{GNEP} is in general intractable, we instead look for local \ac{GNEP} where \eqref{eq:pf:gnep} is satisfied for all $\Ui$ within a local neighborhood of the optimal $\vUs$.

\subsection{Maximum Entropy Dynamic Games}
Although the (global) Nash equilibria is a powerful concept, achieving Nash equilibria requires the assumption that all cost functions are known exactly and that each agent acts rationally by solving the \ac{GNEP} exactly.
However, often times we only have an approximation of the cost function
for the non-ego agents and other agents may only act with approximate solutions of the \ac{GNEP}.
To take these stochasticities into account, we take a \textit{relaxed control} approach and
consider a stochastic control policy $\pi^i(\ui | \vx)$ with the same deterministic dynamics as in \eqref{eq:pf:dyn}.
Let $\ExP{p}{\cdot}$ denote the expectation with respect to a distribution $p$.
We introduce an entropy term to the original objective \eqref{eq:pf:obj_fn} and consider the expected cost under all agent's policies $\pi^i$:
\begin{equation}
    J^i(\pi^i) = \ExP{\pi}{\Phi^i(\vx_T) + \sum_{t=1}^{T-1} \left( l^i(\vx_t, \vu_t) - \alpha H[\pi^i(\cdot | \vx_t)] \right)},
\end{equation}
where $\alpha > 0$ is a temperature term 
and $H[\pi]$ is the Shannon entropy of $\pi$ defined as
\begin{equation}
    H[\pi] = -\ExP{\pi}{\log \pi} = -\int \pi(u) \log \pi(u) \du.
\end{equation}
The resulting \ac{GNEP} is then formulated similarly to the deterministic case, except that we only require that the constraints $h^i$ hold for the mean controls:
\begin{equation} \label{eq:pf_medg:gnep}
\begin{aligned}
    \min_{\pi^i} \quad & J^i(\vU),
    \qquad \textrm{s.t.} \quad h^i(\ExP{\pi}{\vU}) \leq 0.
\end{aligned} 
\end{equation}
With \eqref{eq:pf_medg:gnep}, the task is now to find a stochastic policy $\pi^{i*}$ such that the following holds for all $\pi^i$ within a neighborhood of $\pi^{i*}$:
\begin{equation} \label{eq:pf:gne_nec}
    J^i(\pi^{i*}, \pi^{\negi*}) \leq J^i(\pi^i, \pi^{\negi*}), \quad \pi^i \in \Pi(\pi^{\negi*}),
\end{equation}
where $\Pi(\pi^{\negi*}) \coloneqq \{ \pi^i | h^i(\ExP{\pi}{\vU}) \leq 0 \}$.

\section{Algorithms for Solving Maximum Entropy Dynamic Games}
\label{sec:algos}
In this section, we now derive algorithms for solving the MaxEnt Dynamics Games introduced in the previous section.

\subsection{Unconstrained Dynamic Programming}
\label{sec:me:dp}
We first consider the unconstrained case for simplicity.
Defining the value function for agent $i$ given the policies of other agents $\pinegi$ to be
\begin{equation}
    V^i(\vx) = \inf_{\pii} \left\{ J^i(\vx, \pii, \pinegi) \right\},
\end{equation}
applying dynamic programming results in Bellman's equation:
\begin{equation} \label{eq:dp:raw_bellman}
    V^i(\vx) = \inf_{\pi^i} \Big\{
        \ExP{\pi}{ l^i(\vx, \vu) + {V^i}'( f(\vx, \vu) } - \alpha H[\pi^i(\cdot | \vx)] \Big\}.
\end{equation}
In \eqref{eq:dp:raw_bellman} and below we omit the time index $t$ for nonterminal times for simplicity and use ${V^i}'(f(\vx, \vu))$ to denote the value function for agent $i$ at the next timestep.

It is well known that the optimal policy $\piis$ that solves the infimum in \eqref{eq:dp:raw_bellman} is the Gibbs distribution \cite{wang2020variational, kim2020hamilton}
\begin{equation} \label{eq:dp:opt_pol}
    \piis( \ui | \vx ) = \frac{1}{Z^i} \exp\Big( -\frac{1}{\alpha} \ExP{\pinegi}{ {V^i}'(f(\vx,\vu)) + l^i(\vx, \vu) } \Big),
\end{equation}
where $Z^i$ denotes the partition function
\begin{equation} \label{eq:dp:Z_def}
    Z^i \coloneqq \int \exp\Big( -\frac{1}{\alpha}
    \ExP{\pinegi}{ {V^i}'(f(\vx,\vu)) + l^i(\vx, \vu) } \Big) \dui.
\end{equation}
 Although we have obtained a closed-form expression for $\piis$ in \eqref{eq:dp:opt_pol},
it is defined in terms of $\pinegi$ which is unknown.
Hence, unlike the optimal control case, we must solve a system of equations for each agent to find $\piis$ and $\pinegis$.


\subsection{Unconstrained MELQGames}
\label{sec:me}
In this section, we propose using DDP to solve for Nash Equilibria of MaxEnt dynamic games
and derive the unconstrained MELQGames algorithm similar to iLQGames in \cite{fridovich2020efficient}.
For notational simplicity, we will drop the second-order approximation of the dynamics as in \ac{iLQR} in our description of \ac{DDP}. The dropped second-order dynamics terms can easily be added
back in the derivations below. We refer readers to \cite{mayne1966second, li2004iterative, fridovich2020efficient} for a detailed overview of the vanilla \ac{DDP}, \ac{iLQR} and iLQGames algorithms.

The \ac{DDP} algorithm consists of a forward pass and a backward pass.
The forward pass simulates the dynamics forward in time obtaining a set of 
nominal state and control trajectories $(\bar{\vx}_{0:T}, \bar{\vu}_{0:T-1})$,
while the backward pass solves the Bellman equation with a 2nd order approximation
of the costs and dynamics equations around the nominal trajectories.
The boundary conditions for the value functions $V^i$ for each agent are
obtained by performing a 2nd order Taylor expansion of the terminal costs $\Phi^i$:
\begin{equation} \label{eq:me:term_quad}
\begin{aligned}
    V^i_{\vx\vx,T} &= \Phi^i_{\vx \vx}, & V^i_{\vx, T} &= \Phi_{\vx}^i, & V_T= &= \Phi^i,
\end{aligned}
\end{equation}
where we follow the notation in \ac{DDP} literature by denoting partial derivatives via subscripts.

Additionally, we perform a quadratic approximation of the costs and
linear approximation of the dynamics around a nominal trajectory $(\bar{\vX}, \bar{\vU})$
, where $\vX = [\vx_0, \dots, \vx_T]$ denotes a trajectory of states
\begin{align} 
    l^i(\vx, \vu)
    &\approx l^i
    + \begin{bmatrix}l^i_\vx \\ l^i_\vu\end{bmatrix}\T
      \begin{bmatrix}\vdx \\ \vdu\end{bmatrix}
    + \frac{1}{2}
      \begin{bmatrix}\vdx \\ \vdu\end{bmatrix}\T
      \begin{bmatrix}l^i_{\vx\vx} & l^i_{\vx \vu} \\ l^i_{\vu \vx} & l^i_{\vu \vu} \end{bmatrix}
      \begin{bmatrix}\vdx \\ \vdu\end{bmatrix}, \nonumber        \\
    f(\vx, \vu)
    &\approx f   + f_\vx \vdx + f_\vu \vdu,%
    \label{eq:me:quad_cost_lin_dyn}
\end{align}
where we have defined $\vdx = \vx - \vxbar$ and $\vdu = \vu - \vubar$.

Let $Q^i$ denote the terms inside the expectation in Bellman's equation \eqref{eq:dp:raw_bellman}
\begin{equation}
    Q^i(\vx, \vu) \coloneqq l^i(\vx, \vu) + {V^i}'( f(\vx, \vu) ).
\end{equation}
Then, $Q^i$ for the approximated system is quadratic:
\begin{gather}
    Q^i(\vx, \vu) = \bar{V}^i + \delta Q^i, \\
    \delta Q^i \coloneqq
    \begin{bmatrix}Q^i_\vx \\ Q^i_\vu\end{bmatrix}\T \begin{bmatrix}\vdx \\ \vdu\end{bmatrix}
    + \frac{1}{2}
      \begin{bmatrix}\vdx \\ \vdu\end{bmatrix}\T
      \begin{bmatrix}Q^i_{\vx\vx} & Q^i_{\vx\vu} \\ Q^i_{\vu\vx} & Q^i_{\vu\vu} \end{bmatrix}
      \begin{bmatrix}\vdx \\ \vdu\end{bmatrix}.
\end{gather}
with $\bar{V}^i \coloneqq l^i(\vxbar, \vubar) + {V^i}'(f(\vx, \vu))$.
The derivation and full expressions for the partial derivatives of $Q^i$
are included in \cref{sm:Q_derivs}.
Performing a change of variables $\vdu = \vu - \vubar$
on $\pii$ and taking out terms that not functions of $\delui$, the infimum in the Bellman equation \eqref{eq:dp:raw_bellman} simplifies to
\begin{align}
    &\inf_{\pii} \int \pii(\delui) \bigg\{
        \Big( Q_{\ui}^i + Q_{\ui \vx}^i \vdx
    + Q_{\ui \unegi }^i \ExP{\pinegi}{\vdunegi} \Big)\T \delui \nonumber \\
    &\qquad + \frac{1}{2}\delui\T Q_{u^i u^i}^i \delui + \alpha \log \pii(\delui)
    \bigg\} \ddelui, \label{eq:me:quad_bellman}
\end{align}
Since the expression above is now quadratic in $\delui$ for each agent $i$, 
the coupled systems of equations can be solved as shown in the following lemma.
\begin{lemma}[Optimal MELQGames Policy] \label{thm:me:opt_pol}
The optimal policy $\piis$ (in the Nash equilibria sense) which solves \eqref{eq:me:quad_bellman} for each agent $i$ has the form
\begin{equation} \label{eq:me:opt_pol}
    \pi^{i*} =
    {Z^i}^{-1} \exp\Big(-\frac{\alpha}{2} (\delui - \deluis)\T Q_{\ui\ui}^i (\delui - \deluis) \Big),
\end{equation}
where $\deluis$ is the solution to the following system of equations
\begin{equation} \label{eq:me:opt_pol_coupled_eq}
    0 = \hat{Q}_{\vu\vu} \vdus + Q_{\vu\vx} \vdx + Q_{\vu}, \quad \vdus = \vk + \vK \vdx,
\end{equation}
and the matrix $\hat{Q}_{\vu\vu} \in \Rb^{n_\vu \times n_\vu}$ is obtained by vertically stacking
the row vectors $Q^i_{\ui \vu}$ for each $i \in \{ 1, \dots, N \}$:
\begin{equation}
    (\hat{Q}_{\vu\vu})\T \coloneqq
    \begin{bmatrix}
    (Q^1_{u^1 \vu})\T &
    \hdots &
    (Q^N_{u^N \vu})\T
    \end{bmatrix}\T.
\end{equation}
\end{lemma}
We defer the proof of \cref{thm:me:opt_pol} to the appendix in \cref{sm:proof:thm:me:opt_pol}.
Note that the optimal policy $\pi^{i*}$ is Gaussian distribution with mean $\deluis$ and covariance matrix $\Sigma^i = \alpha (\Quui)^{-1}$, where $\deluis$ has the same expression as in the iLQGames case.
Additionally, as $\alpha \to 0$, $\pi^{i*}$ converges to the delta distribution centered on $\delui^*$.
\begin{figure}[t]
    \centering
    \includegraphics[width=0.7\linewidth]{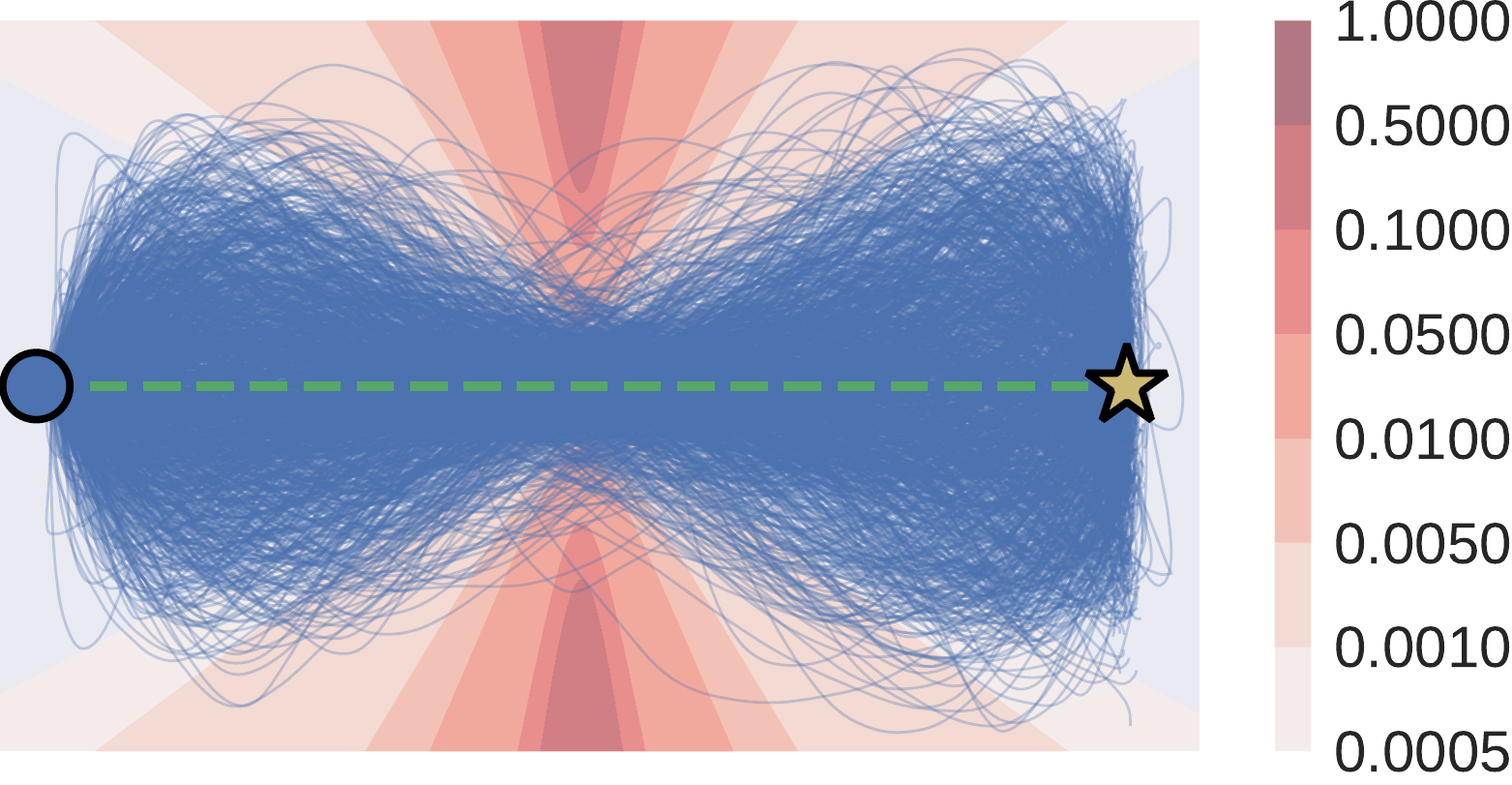}
    \caption{
    Sampled trajectories from the stochastic policy obtained from solving the
    \ac{MaxEnt} Nash equilibria \eqref{eq:me:opt_pol} for a single agent.
    The covariance of the MaxEnt policy reflects the curvature of the cost landscape
    with a tighter distribution of position in the middle than at the ends.
    }
    \label{fig:squeeze_y}
\end{figure}
Information from the value function is incorporated in the policy in both the mean and the covariance
as the covariance is higher along eigenvectors of $\Quui$ which have small eigenvalues.
\cref{fig:squeeze_y} illustrates sampled trajectories from 
the optimal \ac{MaxEnt} Nash equilibria for a single agent with triple integrator dynamics.
The trajectories have a tighter distribution in the middle where
the cost function, shown in the contour plot, has a much higher curvature.
Also, the terminal quadratic cost is higher along the direction of the x-axis,
causing the trajectories near the target (yellow star) to have a lower 
variance along the x-axis.

We next substitute equation \eqref{eq:me:opt_pol} from \cref{thm:me:opt_pol} into the Bellman equation \eqref{eq:me:quad_bellman} to derive the update equations for the value function, shown in the following lemma.
\begin{lemma}[MELQGames Value Function Update] \label{thm:me:val_fn}
Suppose the infimum in the Bellman equation \eqref{eq:dp:raw_bellman}
is solved with policy $\pi^{i*}$ with mean $\vdu = \vk + \vK \vdx$ according
to \cref{thm:me:opt_pol}.
Then, the value function using for agent $i$ has the form
\begin{align} 
    V^i(\vx)
    &= \ExP{\vdu \sim \tilde{\pi}}{
        l^i(\vx, \vu) + {V^i}'( f(\vx, \vu) 
    } - \alpha H[\tilde{\pi}^i], \\
    &=\Big( V^i + V_H^i \Big) + {V_\vx^i}\T \vdx + \frac{1}{2} \vdx\T V_{\vx\vx}^i \vdx ,\label{eq:me:val_fn_update}
\end{align}
where the terms $V^i$, $V_{H}^i$, $V_x^i$ and $V_{xx}^i$ have the form
\begin{align}
    V^i
        &= \bar{V}^i + Q_u^i \vk + \frac{1}{2} \vk\T Q_{\vu\vu}^i \vk
        \label{eq:me:dV}, \\
\begin{split}
    V_H^i &= 
        \frac{\alpha}{2} \Big( \log \abs{ Q_{\ui\ui}^i } - n_{\ui} \log (2\pi \alpha) \Big) \\
        &\quad + \alpha \sum_{j=1, j \not= i}^N
            \tr\left[ (Q^j_{u^j, u^j})^{-1} Q^i_{u^j, u^j} \right],
\end{split} \label{eq:me:dVH} \\
    V_\vx^i
        &= Q_\vx^i + \vK\T Q_{\vu\vu}^i \vk + \vK\T Q_\vu^i + Q_{\vx\vu}^i \vk, \label{eq:me:Vx} \\
    V_{\vx\vx}^i
        &= Q_{\vx\vx}^i + \vK\T Q_{\vu\vu}^i \vK + \vK\T Q_{\vu\vx}^i + Q_{\vx\vu}^i \vK. \label{eq:me:Vxx}
\end{align}
\end{lemma}

Again, we defer the proof of \cref{thm:me:val_fn} to the appendix in \cref{sm:proof:thm:me:val_fn}.

Note that the update rules for $V^i, V^i_\vx$ and $V^i_{\vx\vx}$ are exactly the same as in
iLQGames \cite{fridovich2020efficient}, with the only difference being the addition of the $V^i_{H}$ term resulting from the maximum entropy term.
This term only depends on $\Quui$ and approaches $0$ as $\alpha \to 0$.
Since we take $\Quui$ to be constant during each iteration of MELQGames, $V^i_H$ is not a function of $\vx$.
Consequently, the backward pass of MELQGames can be performed by additionally computing $V_H^i$ and $\Sigma^i$ in the backward pass of iLQGames. Also, note that unlike the optimal control case where we can simplify the value function update \eqref{eq:me:dV}--\eqref{eq:me:Vxx} by substituting the definition of $\vk$ and $\vK$ to cancel out terms, this does not hold in the dynamic games case.
This is because $Q_{\vu\vu}^i$ (Hessian of $Q^i$) does not equal to $\hat{Q}_{\vu\vu}$ (stacked rows of $Q_{\ui \vu}^i$) in general.

\subsection{Constrained MELQGames via Augmented Lagrangian}
\label{sec:constrained_me}
To incorporate constraints, we use the augmented Lagrangian framework
similar to the approach in ALGames \cite{cleac2019algames}, except we choose to use the DDP style unconstrained optimizer as opposed to Newton's method.
Let us denote by $\mathcal{L}^i$ the augmented Lagrangian for agent $i$.
Then,
\begin{equation}
    \mathcal{L}_A^i(\pi, \lambda; \rho) = J^i(\pi) + \frac{\rho}{2}
        \sum_{j} \max\bigg( 0, (h^i)_j(\ExP{}{\vU}) + \frac{\lambda^i_j}{\rho} \bigg)^2.
\end{equation}
where $\rho > 0$ denotes the penalty parameter and $\lambda^i_j > 0$ denotes the Lagrange multiplier corresponding to the $j$th constraint $(h^i)_j$ for agent $i$.
Since each agent may in general have a different number of constraints, the Lagrange multipliers are specific to each agent's constraints.

If the optimal Lagrange multipliers $\lambda^{i*}_j$ were known, then for a sufficiently large value of $\rho$ the minimizer of the augmented Lagrangian $\mathcal{L}_A^i$ would be a local minimum of the constrained problem \cite{nocedal2006numerical, ruszczynski2011nonlinear}.
To update the dual parameters $\lambda^i_j$, we perform the following dual-ascent step
\begin{equation} \label{eq:al:dual_update}
    (\lambda^i_j)^+ \gets \max(0, \lambda^i_j + \rho (h^i)_j(\ExP{}{\vU}).
\end{equation}
where we use the $+$ to denote the new $\lambda$.
Intuitively, the above update can be seen as an approximation to the optimal Lagrange multipliers $\lambda^{i*}_j$. Let $\mathcal{I}_A$ denote the set of active constraints, and suppose that the Lagrange multipliers for all inactive constraints are zero such that
\begin{equation}
  j \not \in \mathcal{I}_A \implies (h^i)_j + \frac{\lambda^i_j}{\rho} < 0 .
\end{equation}
Then, if $\pi$ is a minimizer of $\mathcal{L}_A^i$ then
\begin{equation}
    0
    = \nabla_{\ui} \mathcal{L}_A^i = \nabla_{\ui} J^i
        + \sum_{j \in \mathcal{I}_A} \big[ \rho (h^i)_j + \lambda^i_j \big] \nabla_{\ui} (h^i)_j,
\end{equation}
and hence $(\lambda^i_j)^+ = \rho (h^i)_j + \lambda^i_j$ is a good approximation to the Lagrange multiplier of the original constrained problem.
To maintain dual feasibility of the Lagrange multipliers, we project them by taking the positive part.

Alternatively, the Augmented Lagrangian method can be viewed as projected steepest ascent applied to the dual problem of \eqref{eq:pf:gnep} where $\rho$ plays the role of the step size in \eqref{eq:al:dual_update}.
Indeed, convergence rate results for Augmented Lagrangian methods show that under suitable conditions (ex.  \ac{LICQ} and \ac{SOSC} are satisfied at the local minima), the Lagrange multipliers converge linearly
\cite{nocedal2006numerical, ruszczynski2011nonlinear} with rate
\begin{equation}
    \frac{ \norm{ (\lambda^i_j)^+ - \lambda^{i*}_j } }
    { \norm{ \lambda^i_j - \lambda^{i*}_j } } \leq \frac{M}{\rho},
\end{equation}
for some constant $M$. Hence, larger values of $\rho$ should theoretically result in faster convergence of the Lagrange multipliers.
However, values of $\rho$ that are too large may make the unconstrained problem ill-conditioned and slow down convergence \cite{nocedal2006numerical, ruszczynski2011nonlinear}.
Hence, Augmented Lagrangian methods usually advocate for increasing $\rho$ by some factor $\gamma > 0$ if a sufficient decrease condition on some feasibilty metric such as
\begin{equation} \label{eq:al:metric}
    \mathcal{V} \coloneqq \norm*{ \min\left\{ -(h^i)_j, \frac{\lambda^i_j}{\rho} \right\} },
\end{equation}
has not been met \cite{birgin2014practical}.

The constrained MELQGames algorithm consists of an iterative scheme where we alternate approximately solving the (unconstrained) Nash equilibria for cost functions $\mathcal{L}_A^i$
and updating the Lagrange multipliers $\lambda^i_j$ with \eqref{eq:al:dual_update}.
If the metric $\mathcal{V}$ has not decreased by more than some $\tau \in (0, 1)$ since the last dual update, we increase $\rho$.

We now show in the following lemma that if the augmented Lagrangian MELQGames converges, the solution is a local GNEP.
\begin{lemma} \label{lemma:aug_lang_cvg}
Suppose that the augmented Lagrangian MELQGames converges to a tuple $(\pi^*, \lambda^i_j, \rho)$ for some $\rho > 0$ which satisfies dual feasibility of $\lambda$, \ac{LICQ} and \ac{SOSC}.
Then, $\pi^*$ is a \ac{GNE} for \eqref{eq:pf:gnep}.
\end{lemma}
We defer the proof of \cref{lemma:aug_lang_cvg} to \cref{sm:proof:lemma:aug_lang_cvg}.

\section{Extensions to Modeling Multimodality}
\label{sec:multimodality}
The vast majority of existing works that investigate game theoretic multi-agent interactions
consider unimodal behaviors
\cite{cleac2019algames, fonseca2018potential, fridovich2020efficient,mehr2021maximum, spica2020real,  williams2017autonomous}.
However, this may not be enough when the true value function is not unimodal, a common situation that can arise from nonconvex dynamics, cost functions or constraints.
In this section, we first outline a computational challenge for tackling the problem of multimodality in dynamic games. 
We then propose a computationally efficient multimodal extension to \ac{MELQGames} by reformulating the problem as an incomplete information games with information asymmetry and show how our formulation can additionally be used for Bayesian inference of the latent mode.




In this paper, we focus our attention on multi-agent interactions where multimodality is
induced by uncertainty about multimodal behavior of non-ego agents despite having unimodal best-responses.
We leave extensions to situations where the ego agent has a multimodal best-response by considering approaches such as compositionality \cite{so2021maximum} as future work.

One of the biggest challenges for handling multimodality with multiple agents is that computational costs can quickly become unfeasible due to the combinatorial explosion when accounting for the interactions between all agents.
To see this, consider a game with $N$ agents where each agent has a policy with $A$ different modes.
Evaluating the Nash equilibrium requires evaluating the expectation over all policies for all timesteps $t=0,\dots,T-1$,
resulting in a total of $O(T\, A^N)$ evaluations of the cost function and dynamics.
This combinatorial explosion is computationally intractable for any kind of realtime planning.

\subsection{Multimodal Dynamic Games via Information Asymmetry}
\label{subsec:mme}
We now propose a method for tackling the challenges mentioned above.
To start, we refine our problem setup and identify one agent as the ``ego'' agent.
Given that many applications for planning have the goal of controlling an agent, this is not an unreasonable assumption.
From hereon after, we denote the ego agent with index $1$ and use $\nonego$ to refer to non-ego agents.

Consider a set of $A$ different local (generalized) Nash equilibria with
$\{ (\vX^a, \vU^a) \}_{a=1}^A$
and let $a \in \{ 1, \dots, A \}$ be a discrete latent random variable that
determines the nominal trajectory around which the cost functions and dynamics are approximated.
We now consider an extension of the maximum entropy dynamic game \eqref{eq:pf_medg:gnep}, where all agents except for the ego agent have full knowledge of the value of $a$ and are playing the corresponding optimal policy $\pi^a$.
However, the non-ego agents incorrectly believes that the ego agent knows what the true mode is.
This is now a \textbf{dynamic game with incomplete information} 
since the ego agent only has a \textit{belief} of what the true game is
but does not know what the true dynamics nor cost are.
However, since the non-ego agents know what the true mode is,
their actions act as a \textit{signal} for the ego agent and allow the 
ego-agent to update its belief using this information.

In the theory of games with incomplete information,
a \textit{state of the world} $\omega$
fully defines a possibility of the true cost function and dynamics of the game \cite{maschler2013game}.
Uncertainty over the true state of the world $\omega_*$ implies uncertainty over
what game is being played.
Let $Y$ be a finite set containing all possible states of the worlds
and $p_i : Y \to \Delta(Y)$ be the \textit{belief} of each agent mapping each state of the world $\omega \in Y$ to a probability distribution over $Y$. 

For our problem, let $\omega_{a, s}$ and $\omega_{a, m}$ for $a=1,\dots,A$
denote different \textit{states of the world} such that 
when $\omega_* \in \big\{ \omega_{a, s},\,  \omega_{a, m} \big\}$,
the true cost and dynamics that correspond to the approximated costs and dynamics
around the local (generalized) Nash equilibrium $( \vxbar^a, \vubar^a )$.
The ``$s$'' and ``$m$'' here can be taken to mean ``single'' and ``multi''.
We now define our set $Y$ as
\begin{equation}
    Y \coloneqq \{ \omega_{1, s}, \dots, \omega_{A, s}, \omega_{1, m}, \dots, \omega_{A, m} \}.
\end{equation}

Let $\ind$ denote the indicator function and $p_a$ denote the ego agent's prior over the modes $a$ with support on the set $\{ \omega_{a, m} \}_{a=1}^A$.
We define the beliefs of the ego agent $p_1$ and non-ego agents $p_{\nonego}$ as
\begin{align}
    p_1(\omega) &= \begin{dcases}
        \ind_{\omega_{a, s}}(\omega), &\omega = \omega_{a, s}, \; a \in \{1, \dots, A \} \\
        p_a(\omega), &\omega = \omega_{a, m}, \; a \in \{1, \dots, A \}
    \end{dcases}, \\
    p_{\nonego}(\omega) &= \begin{dcases}
        \ind_{\omega_{a, s}}(\omega), &\omega = \omega_{a, s}, \; a \in \{1, \dots, A \} \\
        \ind_{\omega_{a, s}}(\omega), &\omega = \omega_{a, m}, \; a \in \{1, \dots, A \}
    \end{dcases},
\end{align}
To give intuition to the above game setup, note that each mode $a$ is associated with two world states: $\omega_{a, s}$ and $\omega_{a, m}$.
When the true state $\omega_*$ is at $\omega_{a, s}$, all agents correctly believe that the
state is $\omega_{a,u}$.
However, when $\omega_* = \omega_{a, m}$, the non-ego agents incorrectly believe that the true state is $\omega_{a, s}$ while the ego agent correctly believes that the true state lies is in the set
$\{ \omega_{a, m} \}$ but is unsure which is the correct one.

With the above setup, we can use tools from game theory to gain intuition into the \textit{belief-structure} of this game, namely what each agent believes is true, what each agent believes other agents believe is true and so on \cite{maschler2013game}.
\begin{definition}[Belief Operator] \label{def:belief_operator}
Let $B_i$ denote the \textit{belief operator} such that for any event $A \subseteq Y$
\begin{equation} \label{eq:def:belief}
    B_i A \coloneqq \{ \omega \in Y : p_i(A | \omega) = 1\},
\end{equation}
and the conditional belief $p_i(A | \omega)$ is defined as
\begin{equation}
    p_i( A | \omega) \coloneqq \ExP{p_i(\omega)}{ \mathbbm{1}_{A} }.
\end{equation}
\end{definition}
In other words, $\omega \in B_i A$ means that at $\omega$, event $A$ \textit{obtains} according to
agent $i$'s belief in the sense that at least one of $\omega \in A$ corresponds to the true state.
We next define the concept of \textit{common belief} \cite{maschler2013game}:
\begin{definition}[Common Belief] \label{def:common_belief}
    Let $A \subseteq Y$ be an event and $\omega \in Y$.
    The event $A$ is \textit{common belief} at state $\omega$ if every agent believes that $A$ obtains, every agent believes that every agent believes that $A$ obtains, and so on.
    In other words, for every finite sequence $i_1, \dots, i_L$ of agents:
    \begin{equation}
        \omega \in B_{i_1} \dots B_{i_L} A,
    \end{equation}
\end{definition}
In particular, the event $Y$ is common belief among the agents, i.e., each agent knows that the true world state $\omega$ must be an element of $Y$.

Using \cref{def:belief_operator} and \cref{def:common_belief}, we now precisely define the key property of the dynamic game --- it is common knowledge that the non-ego agents believes that the ego agent knows the mode.
\begin{lemma} \label{thm:mme:common_belief}
Let $E = \{ \omega_{a,s} \}_{a=1}^A$ denote the event that the ego agent knows the correct mode with probability 1.
Then, for all $i \not= 1$, $\omega \in B_i E$, i.e. all non-ego agents believe that the ego agent knows the mode perfectly.
Furthermore, it is common belief that $B_i E_a$ for all $i \not = 1$.
In particular, the ego agent (correctly) believes that all non-ego agents (incorrectly) believes that it knows the correct mode.
\end{lemma}
We delegate the proof of \cref{thm:mme:common_belief} to \Cref{sm:sec:proof:mme:common_belief}.
Furthermore, we refer the interested reader to \citet{maschler2013game} for more information on 
belief spaces and games with incomplete information.
The above lemma shows how this choice of information asymmetry simplifies computation.
Namely, since each non-ego agent believes that the ego-agent knows the mode perfectly, each non-ego agent's control will exactly be the Nash-equilibria control in the unimodal case.

Finally, we introduce the concept of a (local) \textit{Bayesian Nash equilibrium}.
This is a generalization of the classical (generalized) Nash equilibrium to the incomplete information case, where the \textit{expectation} of an agent's cost with respect to its own belief is used \cite{maschler2013game}:
\begin{definition}[Bayesian Nash equilibrium]
    A policy $\pi^*$ is a Bayesian Nash equilibrium if, for all agents $i$,
    \begin{equation}
        \ExP{p_i}{ J^i(\pi^{i*}, \pi^{\negi *}) } \leq \ExP{p_i}{ J^i(\pi^i, \pi^{\negi *}) },
        \quad \pi^i \in \Pi(\pi^{\negi *}),
    \end{equation}
    where the expectation is taken with respect to agent $i$'s belief of the world state $\omega_* \in Y$.
\end{definition}
As a result, when $\omega_* = \omega_{a, s}$,
since all agents know the mode $a$, the optimal $\pi^*$ simply corresponds to the optimal policy $\pi^{a}$ for the generalized Nash equilibrium found previously.
When $\omega_* = \omega_{a, m}$, the non-ego agents believe that the true state 
is still $\omega_{a, s}$ and hence the optimal policy is still $\pi^{\negi, a}$.
On the other hand, for $\omega_* = \omega_{a, m}$, although the ego-agent is uncertain about the true value of $a$, it knows via \cref{thm:mme:common_belief} that the non-ego agents will play $\pi^{\negi, a}$.
Furthermore, the controls of the non-ego agents $u^{\negi}$ give information about what the true value of $a$ is.
Hence, the problem of computing the Bayesian Nash equilibrium $\pi^{i*}$ reduces to the problem of solving a \ac{POMDP}
where the belief space is over the latent variable $a$ and the observations $\vo$ are the observed controls of the non-ego agents.
The conditional dependencies for the POMDP are shown in the graphical model in \cref{fig:pomdp_model}.
\begin{figure}
    \centering
    \includegraphics[width=\linewidth]{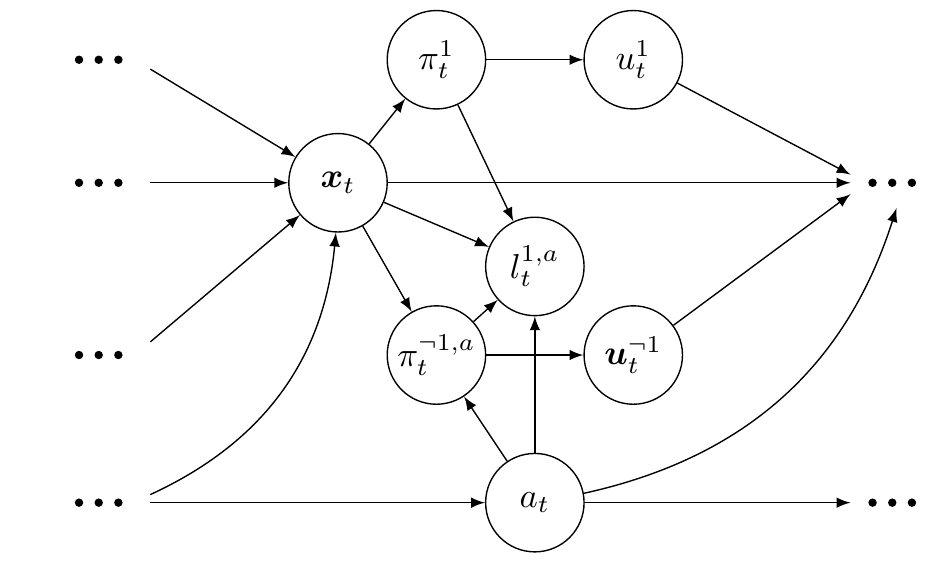}
    \caption{
    Graphical model of the POMDP for the ego agent with corresponding value function \eqref{eq:ego_pomdp}.
    The state $\vx_t$ and non-ego agent controls $\vu_t^{\nonego}$ are observed variables,
    while the computed set of non-ego policies $\pi_t^{\nonego,a}$ is dependent on
    the partially observable discrete latent mode $a$.
    The latent mode $a$ also affects the true cost function $l^{1,a}_t$ and the dynamics for the next state $\vx_{t+1}$.
    }
    \label{fig:pomdp_model}
\end{figure}
By using the standard belief-space approach to solve the \ac{POMDP}, the Bellman equation takes the form
\begin{equation} \label{eq:ego_pomdp}
    V^1(\vx, \vb_t) = \inf_{\pi^1} \Eb_{\pi} \bigg[
        \Eb_{\vb_t} l^{1,a}(\vx, \vu) + \Eb_{\pi^{\nonego}}
            {V^1}'\big( f(\vx, \vu), \vb_{t+1} \big)
    \bigg],
\end{equation}
where $\vb_{t+1}$  refers to the updated belief for the next timestep via Bayesian filtering:
\begin{equation}
    \vb_{t+1}(a) \propto \pi^{\nonego}(\vu^{\nonego} | a) \, \vb_{t}(a).
\end{equation}
While solving for general \ac{POMDP}s is computationally untractable \cite{papadimitriou1987complexity},
there exist computationally efficient methods for solving \ac{POMDP}s
over discrete latent spaces by using a \ac{DDP}-based approach \cite{qiu2020latent}.
However, to simplify the implementation and presentation of this work, we choose to bypass the problem of solving the full POMDP by
making the following assumption.
\begin{assumption}
The ego agent will be informed of the true mode $a$ after one timestep at $t=1$.
\end{assumption}
With this assumption, the value functions for $t \geq 1$ are known and correspond to the value functions $V^{1,a}$ for each mode $a$ of the found local Nash equilibrium.
Hence, the Bellman equation at $t=0$ reads
\begin{align} \label{eq:approx_pomdp}
\begin{split}
V^1
&= \inf_{\pi^1} \Eb_{\vb_0, \pi} \bigg[
        l^{1,a}(\vx, \vu)
        + {V^{1,a}}'\big( f(\vx, \vu) \big)
    \bigg] - \alpha H[\pi^1],
\end{split} \nonumber \\
\begin{split}
&= \inf_{\pi^1} \bigg\{ \Eb_{\pi^i} \Big[
    \frac{1}{2} {u^1}\T \tilde{Q}_{u^1 u^1}^1 u^1
    + \left( \tilde{Q}_{u^1}^1 + \tilde{Q}_{u^1}^1 \vx \right)\T u^i
\Big] \\
&\qquad\qquad + c - \alpha H[\pi^1] \bigg\},
\end{split}\raisetag{1\baselineskip}
\end{align}
where $\vb_0 \in \Delta(Y)$ denotes the ego agent's prior belief of the true mode, $c$ encapsulates all terms that are constant with respect to $u^1$, and
\begin{align}
    \tilde{Q}_{u^1 u^1}^{1} &\coloneqq \ExP{\vb_0}{ Q_{u^1 u^1}^{1,a} },
    \qquad \tilde{Q}_{u^1 \vx}^{1} \coloneqq \ExP{\vb_0}{ Q_{u^1 \vx}^{1,a} }, \\
    \tilde{Q}_{u^1}^{1} &\coloneqq \ExP{\vb_0}{
        Q_{u^1}^{1, a} - Q_{u^1 u^1}^{1,a} \bar{u}^{1,a} - Q_{u^1 \vx}^{1,a} \bar{\vx}^a }.
\end{align}
Consequently, the optimal policy $\pi^1$ corresponds to the Gibbs distribution and is solved in the 
same way as in \ac{MELQGames} \cref{thm:me:opt_pol}.

\textbf{Choices of the prior distribution: }
There are many valid choices for the prior distribution on the modes $\vb_0$.
If no prior information on the non-ego agents is known, one choice is to choose $\vb_0$ as the maximum entropy distribution that minimizes the sum of the value functions for each agent:
\begin{equation} \label{eq:fair_mode_prior}
    \vb_0(a) = \inf_{p} \ExP{a \sim p}{ \sum_{i=1}^N V^{i,a}(\vx_0) } - \alpha H[p].
\end{equation}
With this prior, modes that result in lower costs for all agents will have higher probabilities over
modes that have high costs for all agents.
However, in the context of planning for an ego agent, this may favor local minima which are bad for the ego agent when better alternatives exist.

\begin{figure*}[t]
\centering
\includegraphics[width=0.9\textwidth]{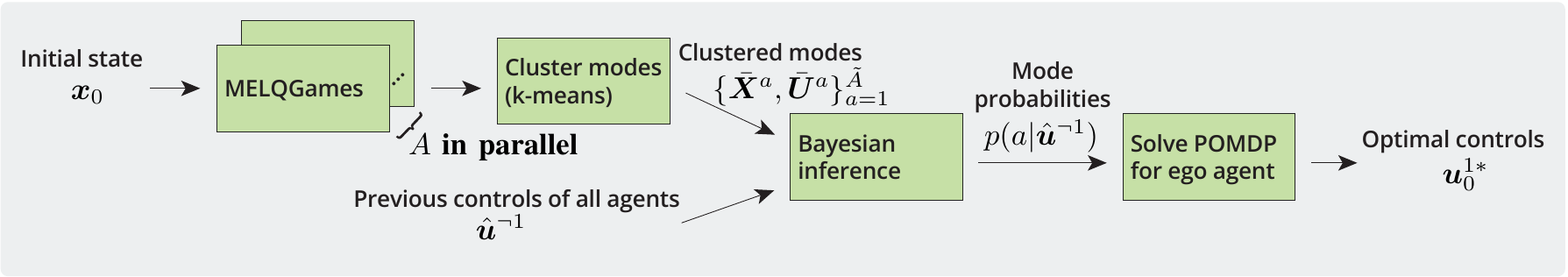}
\caption{Schematic diagram of the proposed MPC \ac{MMELQGames} algorithm presented in \eqref{alg:mme_mpc} and \eqref{alg:inference_mpc}.
}
\label{fig:alg_diagram}
\end{figure*}
Alternatively, we can bias the selection of modes to ones which are favorable to the ego-agent by choosing $\vb_0$ as the maximum entropy distribution that minimizes the sum of the ego agent's value functions
\begin{equation} \label{eq:ego_biased_mode_prior}
    \vb_0(a) = p(a|\vx_0) = \inf_{p} \ExP{a \sim p}{ V^{1,a}(\vx_0) } - \alpha H[p].
\end{equation}

\subsection{Latent Mode Bayesian Inference via MaxEnt}
\label{subsec:bayesian_inference}
In a MPC context, we can improve on our choices of the prior since the controls of the non-ego agents from previous timesteps act as signals for the true mode $a$.
By observing them, the ego agent can form a more accurate estimate of $\vb_0$.
For simplicity, we assume that the observations of non-ego agent's controls are noiseless, though this can be generalized to the noisy case in a Bayesian fashion by incorporating an observation model.
Suppose that the controls $\hat{\vu}^{\nonego} \coloneqq \vu_{-k:-1}^{\nonego}$ and optimal policies $\{ \pi^{\nonego, a} \}_{a=1}^A$ for all non-ego agents from the past $k$ timesteps are known.
Then, the posterior on those previous controls can be used for our estimate of $\vb_0$:
\begin{equation} \label{eq:inference:posterior}
    \vb_0(a)
    \coloneqq p( a | \hat{\vu}^{\nonego} )
    \propto p(a) \prod_{t=-k}^{-1} p(\vu_{t}^{\nonego}|a) ,
\end{equation}
where $p(a)$ is chosen to be the prior \eqref{eq:ego_biased_mode_prior} from the previous subsection
that biases towards modes that are favorable to the ego-agent.
With this choice of prior, modes which describe the controls of non-ego agents and are beneficial to the ego-agent will have high probability.

When computing the posterior in \eqref{eq:inference:posterior},
it is assumed that all of the controls in $\hat{\vu}^{\nonego}$ correspond to the same mode $a$.
While one could simply use all of the collected observations to compute the posterior,
this choice is invalid in the case that non-ego agents switch to a different mode somewhere in the middle of the collected observations.
To solve this issue, we can make the less restrictive assumption that only the last $k$ controls from non-ego agents correspond to the same mode.

\section{Game Theoretic Model Predictive Control}
\label{sec:gt_mpc}
In this section, we summarize the MMELQGames algorithm which combines the
solution to the information asymmetric game setup described in \cref{subsec:mme}
with the latent mode Bayesian inference in \cref{subsec:bayesian_inference} and propose a receding horizon game-theoretic planner that can reason about and infer multiple hypothesis in \cref{alg:mme_mpc}, \cref{alg:find_modes}, \cref{alg:inference_mpc} and summarized in \cref{fig:alg_diagram}.
\Cref{alg:mme_mpc} can be summarized as first solving for a set of local Nash equilibrium in \cref{alg:find_modes},
then using the solutions as inputs to solve a POMDP for the ego agent with uncertainty over the mode in \cref{alg:inference_mpc}.

\begin{algorithm}[t]
\caption{MMELQGames}
\begin{algorithmic}[1]

\State Initialize $\{ \pi \}_{a=1}^{\tilde{A}}$ with random policy.
\State Initialize buffer of non-ego control observations $\mathcal{U}$.
\MRepeat
    \State $\{ (\vX^a, \pi^a) \}_{a=1}^{\tilde{A}} \gets \textrm{ Find Modes}$
    \State $\pi^{1} \gets \textrm{ Solve Ego Agent POMDP}$
    \State Execute $\pi^{1}$ by sampling or taking the mean control
    \State Observe non-ego controls $\vu^{\nonego}$ and add to $\mathcal{U}$
    \State Shift $\pi$ by one timestep for warmstarting next iteration
\EndRepeat

\end{algorithmic}
\label{alg:mme_mpc}
\end{algorithm}
\begin{algorithm}[t]
\caption{Find Modes}
\begin{algorithmic}[1]
\Require Number of modes $A$,
Number of iterations $K$,
control estimates $\{ \vubar^a \}_{a=1}^A$,
Penalty parameter $\rho$,
Feasibility progress $\tau > 0$,
Penalty increase $\gamma > 0$,
Thresholds $\epsilon>0$

\State Initialize $\lambda^i_j = 0, \mathcal{V}^a_0 = \infty, \rho^a = 1$
\State Compute nominal mean trajectory $\{ \bar{\vX}^a \}$ using $\{ \bar{\vU}^a \}$.
\For{$k=1$ to $K$}

\For{$a=1$ to $A$ \textit{in parallel}}
    \State $\pi^{i,a*}, V^{i,a}_{x}, V^{i,a}_{xx} \gets$ Backward pass \eqref{eq:me:opt_pol}, \eqref{eq:me:dV}--\eqref{eq:me:Vxx}
\EndFor

\For{$a=1$ to $A$ \textit{in parallel}}
    \State $(\bar{\vX}^a, \bar{\vU}^a) \gets $ Line Search on $\max_i \norm{ J^i_{\ui} }$
\EndFor

\State Compute $\mathcal{V}^i_k$ using \eqref{eq:al:metric}
\If{ $\max_i \norm{ J^i_{\ui} } \leq \epsilon$ } \Comment{Unconstrained Converged}
    \If{ $\mathcal{V}^a_k \leq \epsilon$ } \Comment{Constraints Converged}
        \State Save $(\bar{\vX}^a, \bar{\vU}^a, \pi^{u, a*}, V_x^{i,a}, V_{xx}^{i,a})$ and reinitialize mode $a$ by sampling from $\pi$.
    \EndIf
    
    \If{ $\mathcal{V}^a_k > \tau \mathcal{V}^a_{k-1}$ } \Comment{Insufficient Decrease}
        \State $\rho^a \gets \gamma \rho^a$
    \EndIf
\EndIf
\EndFor
\State $\{ \bar{\vU}^a \}_{a=1}^{\tilde{A}} \gets \textrm{ Cluster } \{ \bar{\vU}^a \}$
\end{algorithmic}
\label{alg:find_modes}
\end{algorithm}
\begin{algorithm}[t]
\caption{Solve Ego Agent POMDP}
\begin{algorithmic}[1]
\Require Observed controls from past $k$ timesteps $\vu_{-k:-1}^{\nonego}$,
Value functions $V^{i,a}$,
Optimal Policies $\{ \pi^{\nonego, a} \}_{a=1}^{\tilde{A}}$ for non-ego agents

\State $p(a) \gets $ Compute prior with \eqref{eq:fair_mode_prior} or \eqref{eq:ego_biased_mode_prior}
\State $\vb_0 \gets $ Compute initial belief with \eqref{eq:inference:posterior}
\State $\pi^1 \gets $ Solve full \eqref{eq:ego_pomdp} or approximate \eqref{eq:approx_pomdp} POMDP

\end{algorithmic}
\label{alg:inference_mpc}
\end{algorithm}

\Cref{alg:find_modes} starts by first finding a set of modes that each solve the local maximum entropy generalized Nash equilibrium problem \eqref{eq:pf_medg:gnep}.
Moreover, to obtain the optimal policies for the non-ego agents from previous timesteps,
we solve starting from the earliest timestep which we wish to condition our posterior computation \eqref{eq:inference:posterior} on, warm-starting this computation
from the previous iteration's solution if available.
To solve for each mode, we use the \ac{MELQGames} algorithm.
More specifically, we first compute the nominal state trajectories $\{ \vxbar^a \}_{a=1}^A$ as well as derivatives of the cost functions $l^i$ and dynamics $f$ along $(\vxbar^a, \vubar^a)$ for each mode.
Then, we solve for the backward pass by computing the optimal policy \eqref{eq:me:opt_pol} and updating the value function \eqref{eq:me:val_fn_update} for each timestep starting from the terminal time $T$.
Finally, we update the nominal controls using a line search on the maximum of the norm $\max_i \norm{J^i_{\ui}}$.
We exploit the fact that each local Nash equilibrium can be computed independent by solving for $A$ local generalized Nash equilibrium \textit{in parallel}, where $A$ is chosen to be the number of cores available for parallel processing on the CPU.
This iteration of backward and forward pass continues for each mode until this norm is smaller than some threshold $\epsilon$, upon which it is saved into a set of converged modes before being reinitialized from the current $\pi$.
After the maximum number of iterations has been reached, we perform clustering via k-means to remove potential duplicate modes before returning the reduced set $\{ (\vxbar^a, \vubar^a)_{a=1}^{\tilde{A}} \}$ of unique Nash equilibrium.

Then, in \cref{alg:inference_mpc}, we use the ego agent's value function, optimal policies for the non-ego agents $\pi^{\negi,a}$ and
observed controls from non-ego agents to perform Bayesian inference on the mode $a$ to compute our belief prior $\vb_0$ via \eqref{eq:inference:posterior}.
Consequently, we use $\vb_0$ and $\pi^{\negi,a}$ to solve the POMDP \eqref{eq:ego_pomdp} and obtain the optimal policy $\pi^{1,*}$.
This can be done by either solving the full POMDP \eqref{eq:ego_pomdp} using the method from \cite{qiu2020latent}
or by solving the approximated one-step POMDP \eqref{eq:approx_pomdp}.

Finally, we either execute a random sample from $\pi^{1,*}$ or use the mean control before repeating the algorithm again in a receding horizon fashion.

\section{Connections to Related Works}
\label{sec:connections}
\subsection{Game-Theoretic Planning}
There are two main approaches to solving for game-theoretic equilibria: 
\ac{IBR}-based and direct methods.

\textbf{IBR-based methods: }
\ac{IBR}-based methods rely on either the Jacobi decomposition or Gauss-Seidel decomposition to decompose the problem of finding first-order stationary points into separate optimal control problems for each agent \cite{britzelmeier2019numerical, williams2017autonomous}.
These methods optimize the controls for one agent while keeping the controls for other agents fixed, repeating this process iteratively until convergence.
By holding controls, \ac{IBR} assumes that
$
  \frac{ \partial \vunegi_{t+1} }{ \partial \ui_t }
  = \frac{ \partial \vunegi_{t+1} }{ \partial \vx_{t+1} }
  \frac{ \partial \vx_{t+1} }{ \partial \ui_t }
  = 0%
$.
As a result, IBR lacks the term below key to describing inter-agent behavior arising from asymmetric inter-agent coupling terms:
\begin{equation}
    \frac{\partial J^i}{\partial \ui_t}
    =  \frac{\partial J^i}{\partial \vx_{t+2}} \cdot \left(
        \frac{\partial \vx_{t+2}}{\partial \vunegi_{t+1}}
        \frac{\partial \vunegi_{t+1}}{ \partial \vx_t }
        \frac{\partial \vx_t}{ \partial \ui_t }
        + \dots
        \right).
\end{equation}
To remedy this, approaches from \cite{spica2020real, wang2019game, wang2020multi} propose the SE-IBR algorithm which extends \ac{IBR} by
accounting for some of missing information using the sensitivity nformation from the Lagrange multipliers of active constraints.
However, their approach relies on a particular structure of the cost function.
Additionally, the convergence of \ac{IBR}-based methods is not well understood \cite{facchinei2010generalized} and can potentially require many trajectory optimization iterations before convergence.

\textbf{Direct methods: }
Direct methods solve for the coupled system of equations resulting from the first-order stationary conditions directly.
Examples include methods based on Newton's method \cite{di2019newton, di2020first, cleac2019algames} and \ac{DDP} \cite{fridovich2020efficient, schwarting2021stochastic}.
The proposed \ac{MMELQGames} algorithm falls under the category of \ac{DDP}-based methods as it solves a quadratic approximation of the dynamic game during each timestep.
However, unlike previous methods, our method focuses on the constrained \ac{MaxEnt} version of the game and also incoporates multimodality via a novel formulation of an incomplete information with information asymmetry.

\subsection{Multimodal Planning}
There are relatively few works that explore the concept of multimodality within the context of dynamic games. 
This idea is explored in \cite{laine2021multi}, but requires the multimodality to be known a priori and explicitly encoded into the model.
One series of works \cite{ivanovic2019trajectron, ivanovic2018generative, schmerling2018multimodal} approach this problem by training a generative model offline conditioned on a discrete latent variable to learn the trajectories of other agents. While this approach does take multimodality into account, it is not done in a game-theoretic context.

The work in this paper is most similar to \cite{qiu2020latent},
which also solves a POMDP with a discrete latent variable.
However, in their work, the discrete latent variable must be prespecified while the multimodality is discovered in this work.
Also, their work only looks at optimality for a single agent while
in our work we focus on solving for generalized Nash equilibria, with the POMDP being only one piece of the algorithm.

\subsection{Bayesian Inference and Laplace Approximation}
The maximum entropy term transforms the difficulty of each agent's optimization problem into that of sampling from the optimal $\pi^i$ and computing the normalization term $Z^i$, a challenge shared by many techniques for Bayesian inference.
Some popular techniques for tackling Bayesian inference include \ac{MCMC} and \ac{VI} \cite{andrieu2003introduction, salimans2015markov}, though these methods are usually much slower and are computationally intractable for planning on real time. We do note however a series of works \cite{okada2020variational, wang2021variational} which successfully apply variational inference via sampling-based optimization.
Laplace Approximation \cite{daxberger2021laplace} is an approximation framework for Bayesian inference by finding a Gaussian approximation to a probability density function 
MELQGames can be viewed as applying the Laplace Approximation to $\pii_t$ at each timestep $t$.

\section{Experiments}
\label{sec:simulation}
In this section, we compare the proposed \ac{MMELQGames} algorithm against the IBR and MELQGames algorithms on two examples which illustrate the capabilities of MMELQGames in recovering a rich set of multi-agent interactions.
The reader is encouraged to view the videos included in the supplementary material which
showcase the behavior of \ac{MMELQGames} compared to \ac{MELQGames} and \ac{IBR}.

    
\begin{figure}
    \centering
    \includegraphics[width=0.49\linewidth]{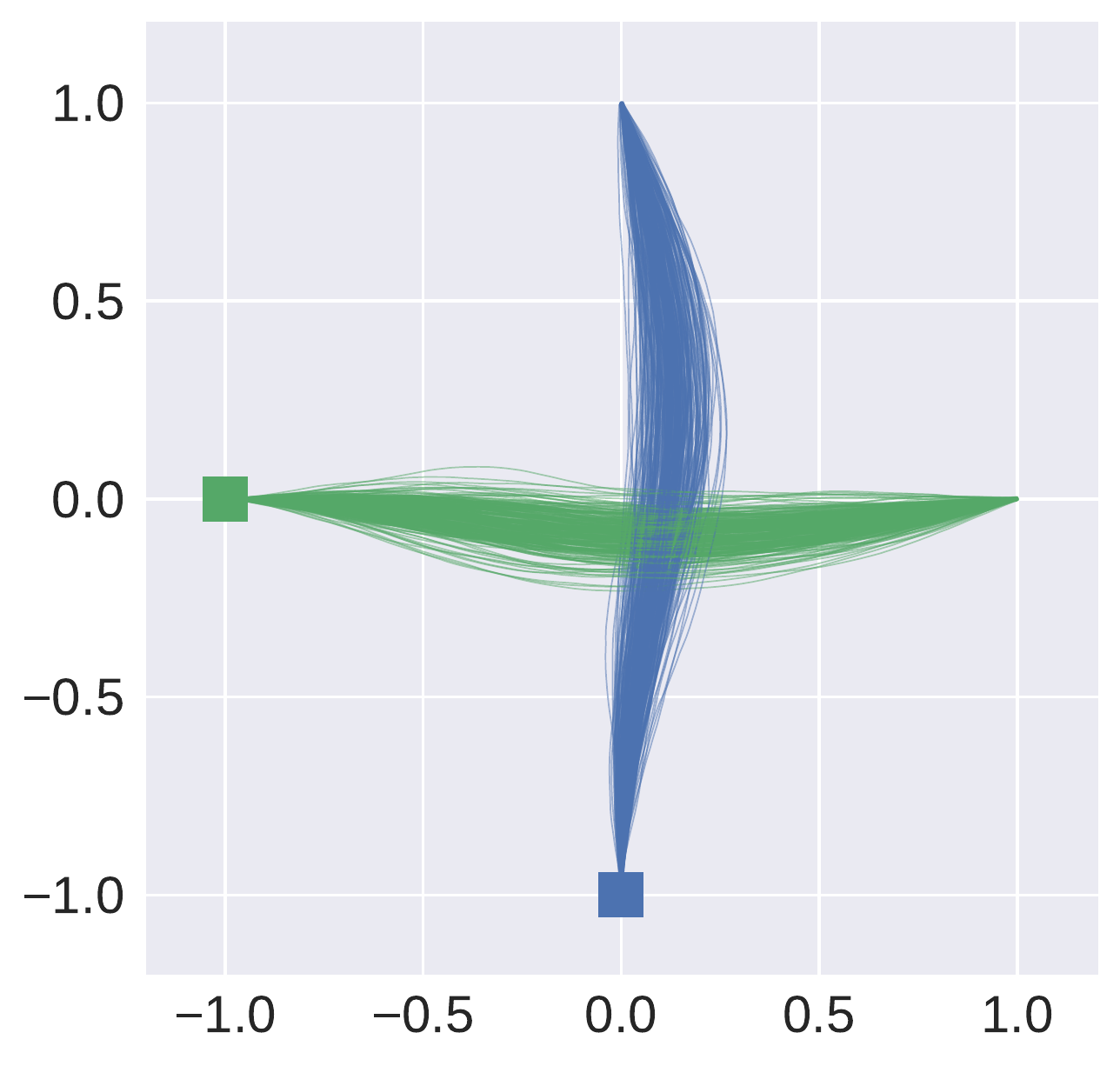}
    \includegraphics[width=0.49\linewidth]{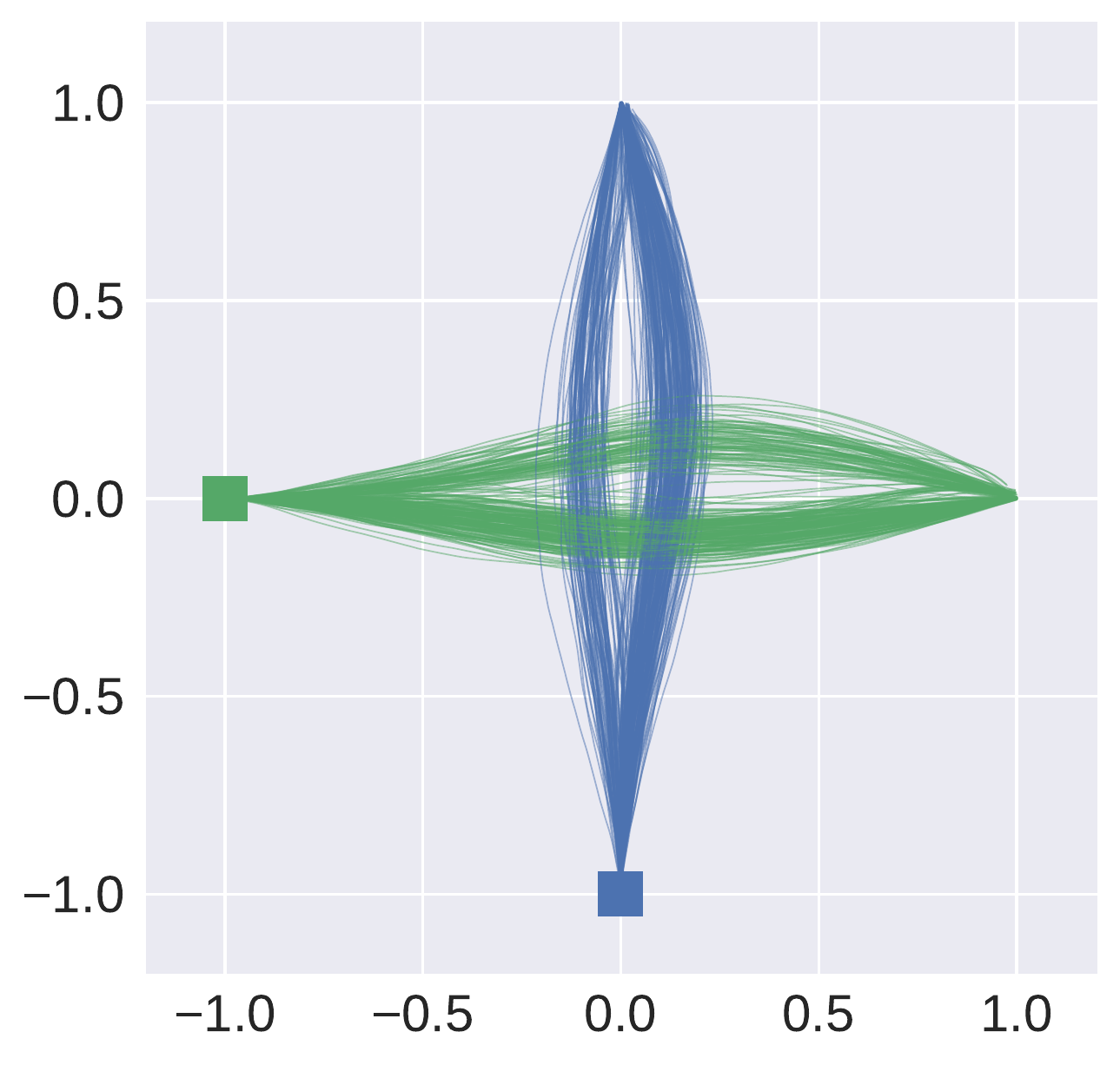}
    \caption{
    Comparison of realizations of the stochastic policy obtained from
    \textbf{M}ELQGames (\textbf{left}) and \textbf{MM}ELQGames (\textbf{right}).
    For \textbf{MM}ELQGames, the mode $a$ is sampled using the prior from \eqref{eq:fair_mode_prior}.
    }
    \label{fig:sim:swap_2}
\end{figure}
\begin{figure}
    \centering
    \includegraphics[width=0.49\linewidth]{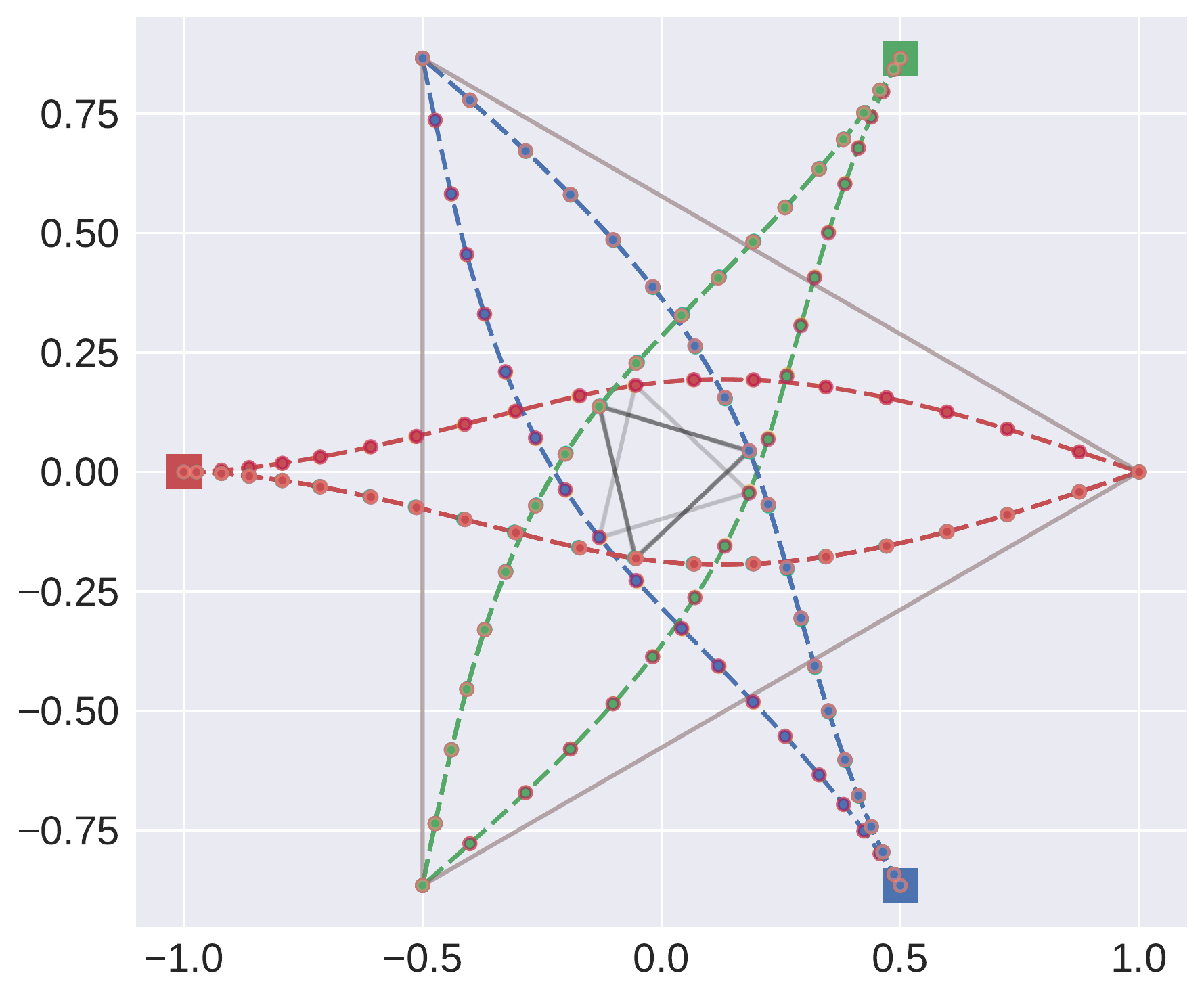}
    \includegraphics[width=0.49\linewidth]{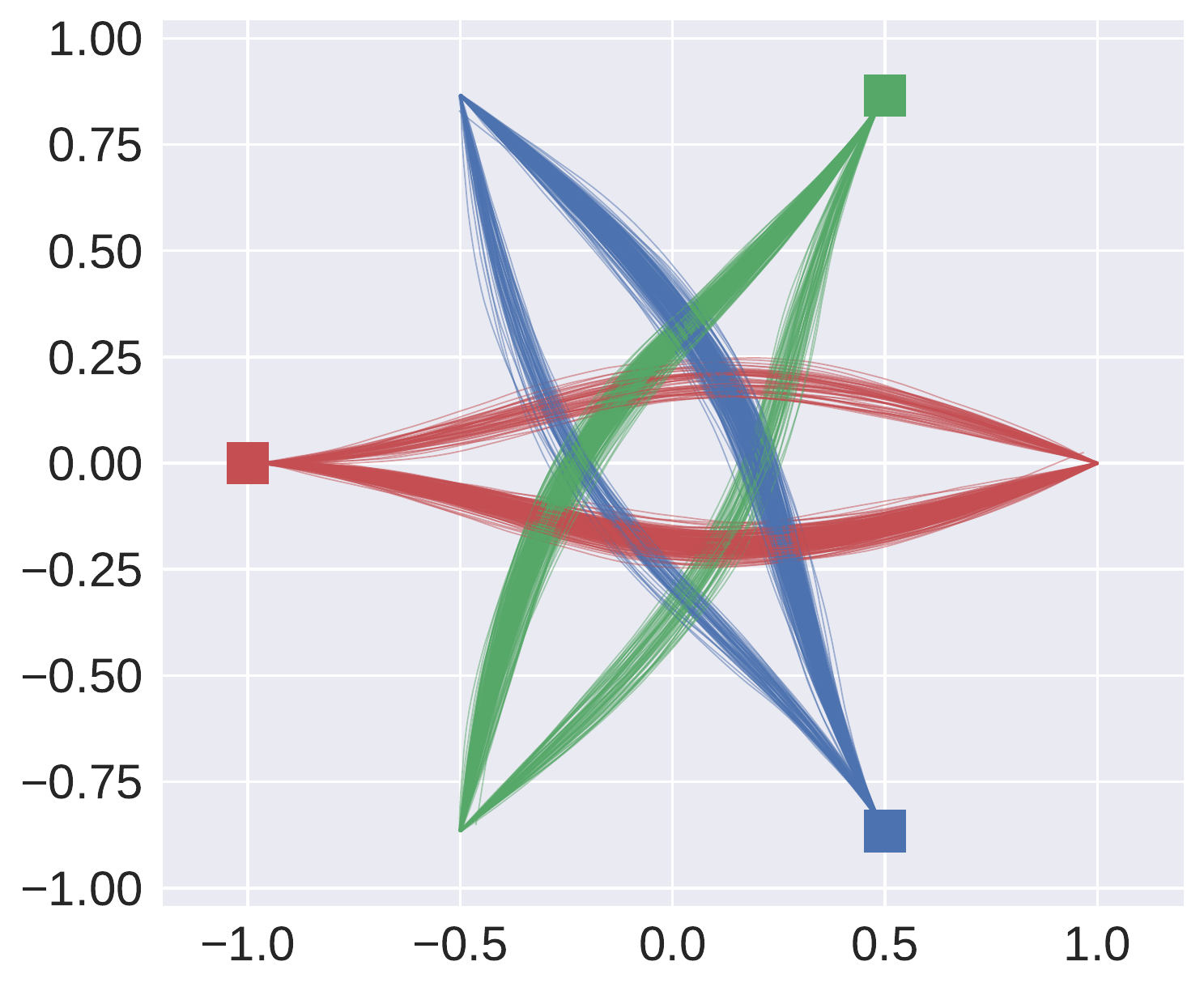}
    \caption{MMELQGames for the three agents case, where the algorithm is able to find both Nash equilibria.
    The \textbf{left} plot shows the planned mean policies for each mode.
    The square denotes the initial position for each agent, with lines indicating that positions are from the same timestep.
    The \textbf{right} plot shows realizations of the stochastic policy.
    Again, the prior from \eqref{eq:fair_mode_prior} is used for sampling the mode $a$.
    }
    \label{fig:sim:swap_3}
\end{figure}

{
\newcommand{\mygriditem}[1]{%
\begin{minipage}[b]{0.41\linewidth}%
    \includegraphics[width=\textwidth]{#1}%
\end{minipage}%
}
\newcommand{\myrowlabel}[1]{%
\parbox[b]{.95\linewidth}{\centering #1}%
}

\begin{figure}[t]
    \centering
    \hspace{.05\linewidth}
    \parbox[b]{.41\linewidth}{\centering Overtake Above}
    \parbox[b]{.41\linewidth}{\centering Overtake Below}\\
    \vspace{0.3em}
    
    \mygriditem{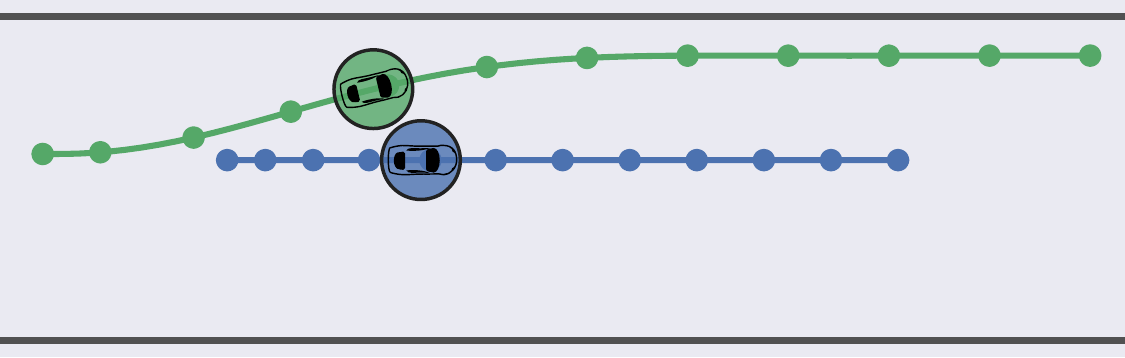}
    \mygriditem{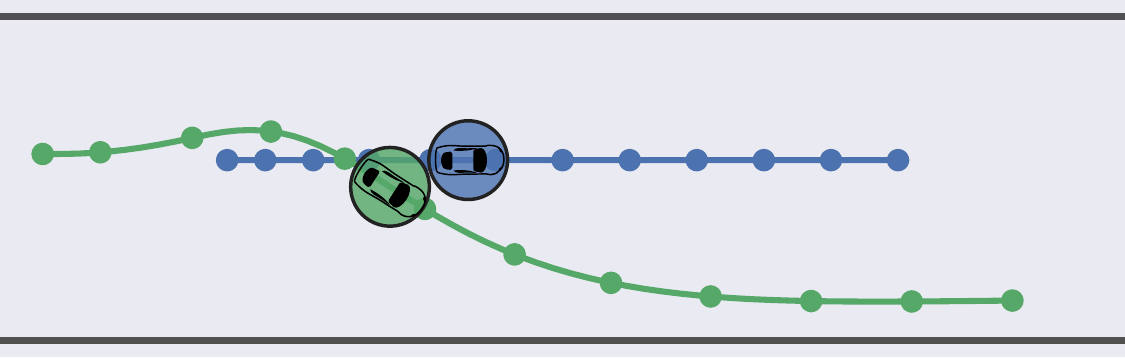}\\
    \vspace{-0.2em}%
    \myrowlabel{IBR}
    
    \vspace{0.6em}
    \mygriditem{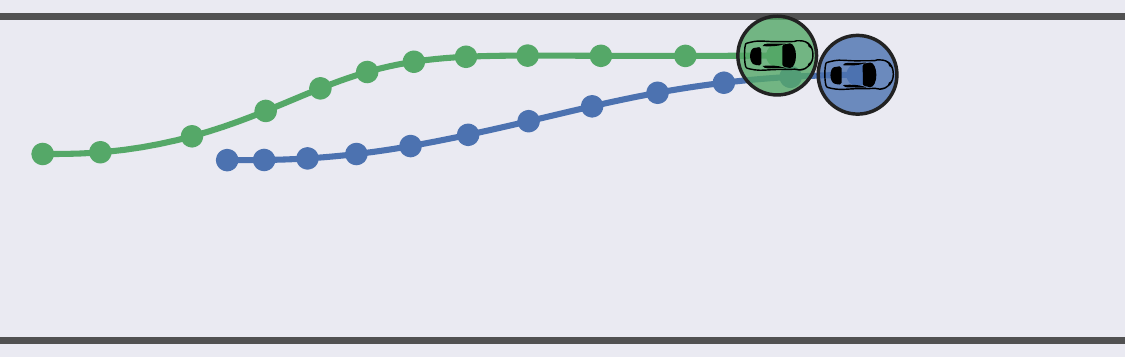}
    \mygriditem{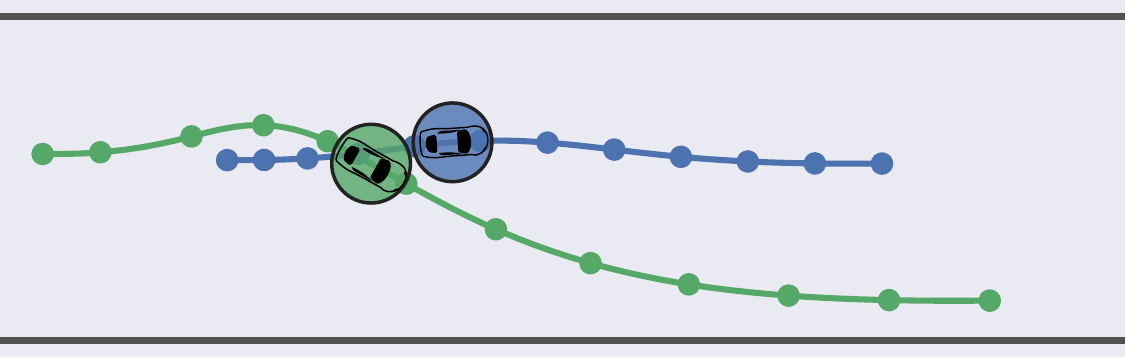}\\
    \vspace{-0.2em}%
    \myrowlabel{iLQGames}
    
    \vspace{0.6em}
    \mygriditem{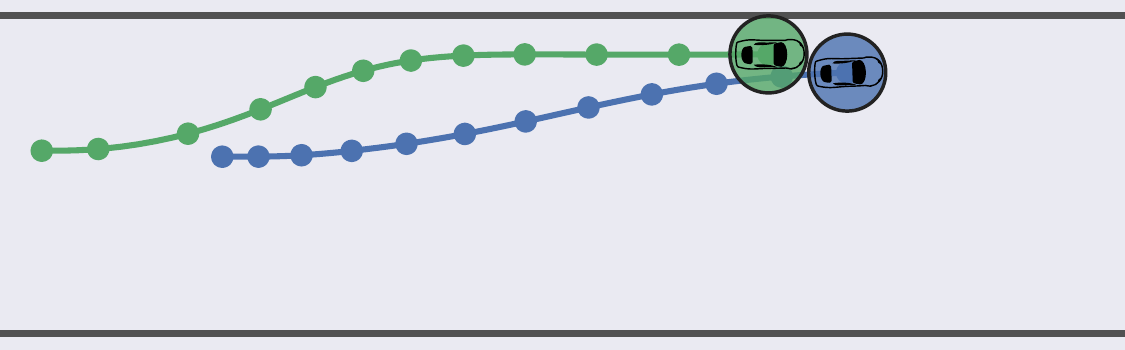}
    \mygriditem{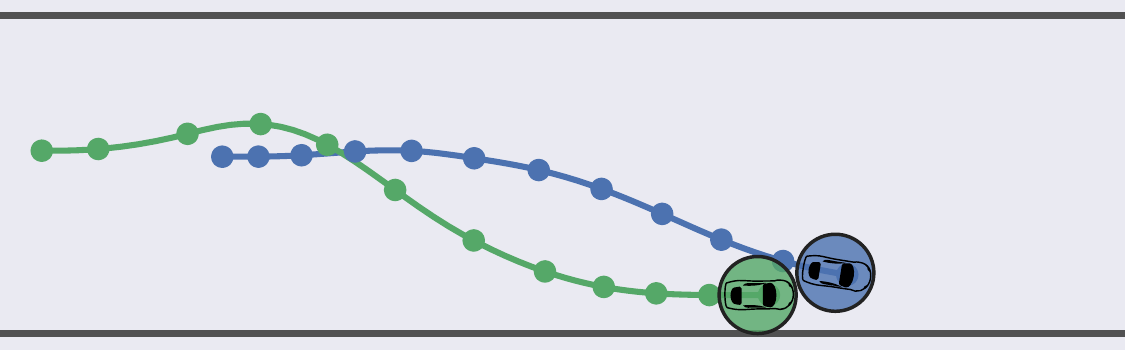}\\
    \vspace{-0.2em}%
    \myrowlabel{MMELQGames}
    
    \vspace{0.1em}
    \caption{Comparison of the trajectories resulting from running IBR (\textbf{top}), iLQGames (\textbf{middle}) and MMELQGames (\textbf{bottom}) in MPC for the lead (blue) agent.
    IBR is used for the rear (green) agent in all cases.
    We test on the case where the rear agent tries to overtake above (\textbf{left}) and below (\textbf{right}) the lead agent.
    Colored circles are added every fifth timestep to indicate the position of each agent.
    Black circles indicating the radius of each agent denote the timestep where the agents are \textbf{closest} to each other.
    }
    \label{fig:sim:overtake}
\end{figure}
}

\subsection{Multi-agent collision avoidance}
We begin with a simple collision avoidance game between agents with unicycle dynamics to highlight the multimodality of our algorithm.
Each agent's objective function composes of a quadratic task and a soft cost for colliding with other agents.
As there is no ego agent in this example, we assume that the ego agent knows the mode $a$ from the start and choose the prior for $a$ to be of the form \eqref{eq:fair_mode_prior}.
The results are shown in \cref{fig:sim:swap_2}.
While \citet{mehr2021maximum} claims that the stochastic policy from MELQGames is able to 
result in multimodal behaviors, our results show that this is not the case.
Since the mean control from MELQGames is equivalent to the linear feedback controller from iLQR, the resulting behavior should also result in a ``tube'' around the mean trajectory, which is indeed what we see on the top row of \cref{fig:sim:swap_2}.
On the other hand, the MMELQGames algorithm is able to properly handle the multimodality of this example and sucessfully find both local Nash equilibria of the game.

\begin{figure}
    \centering
    \begin{minipage}[b]{.49\linewidth}
        \includegraphics[width=\linewidth]{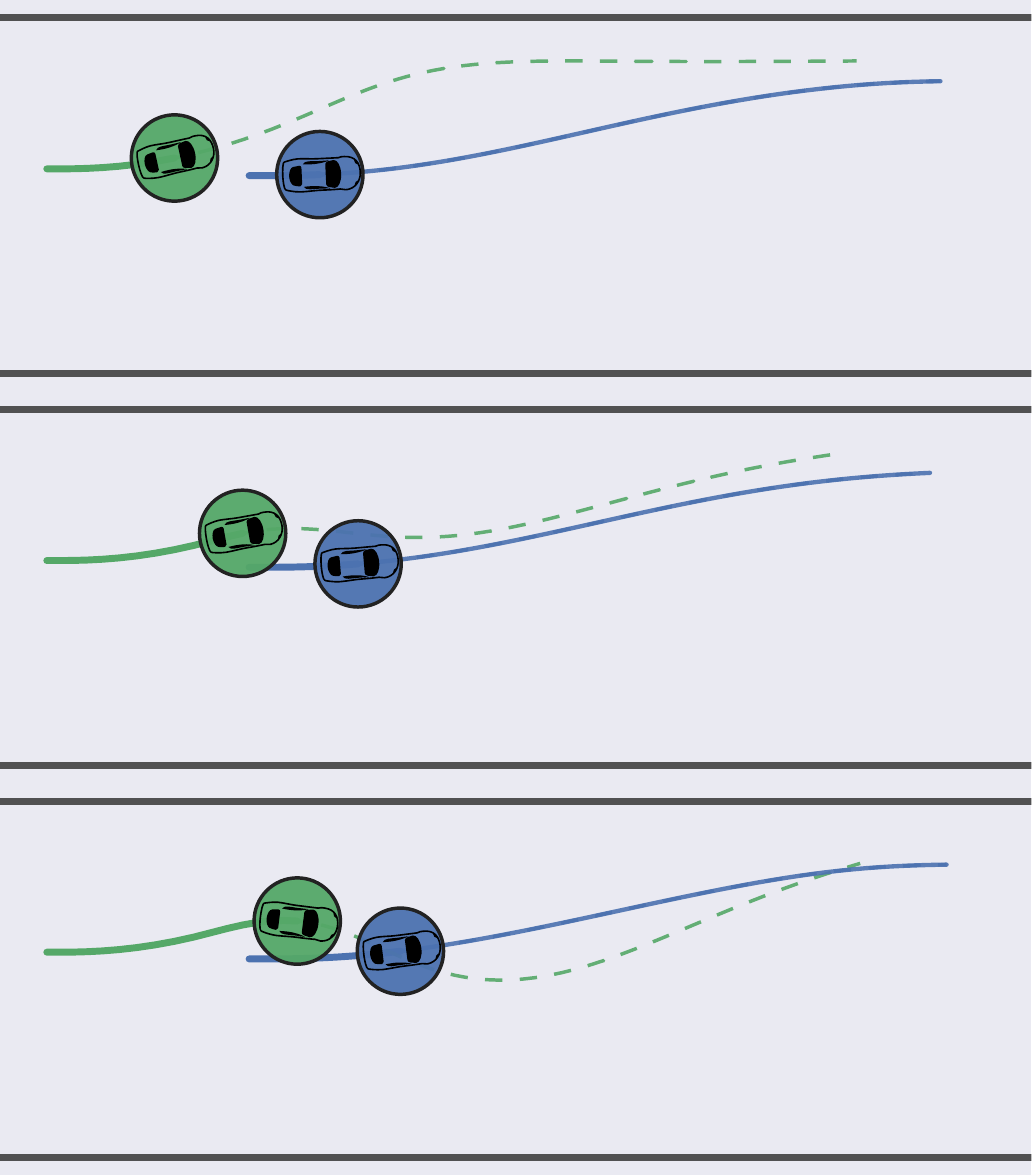}
        \subcaption{iLQGames}
    \end{minipage}
    \begin{minipage}[b]{.49\linewidth}
        \includegraphics[width=\linewidth]{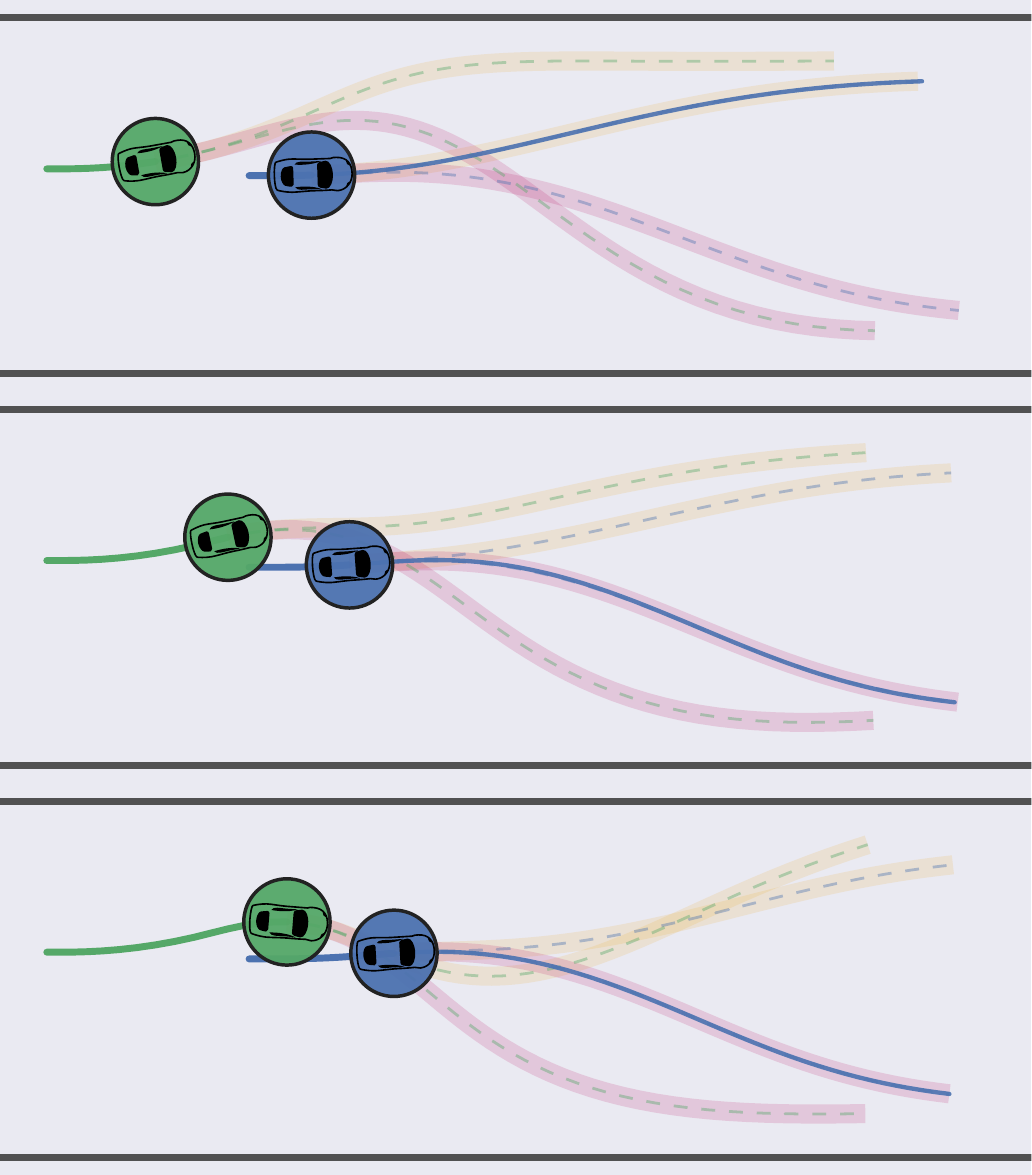}
        \subcaption{MMELQGames}
    \end{minipage}
    
    \caption{
    Snapshot of iLQGames (\textbf{left}) and MMELQGames (\textbf{right}) at three timesteps after the rear agent starts turning towards the bottom mode.
    The ego agent's planned trajectory is shown as a solid blue line,
    with the predicted trajectories of the non-ego agent shown as a dashed line.
    iLQGames incorrectly predicts that the rear agent will overtake from the top due to the algorithm being stuck in a local minima, while MMELQGames is able to successfully infer that the mode has changed.
    For MMELQGames, the different colors indicate the different modes.
    }
    \label{fig:sim:me_wrong}
\end{figure}

\subsection{Game Theoretic Autonomous Racing}
We next consider a two-agent racing scenario where the lead vehicle has index $i=1$ and the rear vehicle has index $i=2$.
The objective of each agent is to maximize the difference between its own progress $s^i$ and the progress of the other agent $s^{\negi}$ (first term)
under quadratic control cost (second term), subject to the following
\begin{align}
    \min_{u^i} &\quad (s^{\negi} - s^i)
        + \sum_{t=1}^{T-1} \frac{1}{2}{u^i_t}\T R u^i_t, \\
    \textrm{s.t.} &\quad h^i_{\textrm{track}}(\vX, \vU) \leq 0, \\
    &\quad \mathbbm{1}_{s^i < s^{\negi}} h^i_{\textrm{collision}}(\vX) \leq 0, \\
    &\quad h^i_{\textrm{velocity}}(\vX) \leq 0.
\end{align}
where each vehicle is modeled as a unicycle with linear and angular accelerations as controls.

To make the problem interesting, the lead vehicle has a lower maximum velocity than the rear vehicle to allow the rear vehicle the opportunity to overtake.
To maintain the lead, the lead vehicle must block the rear vehicle from overtaking its position.
Furthermore, to more closely model racing in real life, the collision constraint is asymmetric with the burden of respecting the constraint placed solely on the rear vehicle.
We compare a MPC version of the IBR, iLQGames and MMELQGames algorithms for the lead vehicle by planning in a receding horizon fashion.
Note that MELQGames is equivalent to iLQGames if the mean controls are used instead of sampling from the stochastic policy.
In all cases, the rear agent solves for its controls using IBR.
The width of the track means that the rear agent can choose to overtake the lead vehicle by going either above or below
resulting in multimodality.
We test both modes in our experiments, with the results for \ac{MMELQGames} shown in \cref{fig:eyecandy} and a comparison of the algorithms shown in \cref{fig:sim:overtake}.

As expected, IBR is unable to take into account the asymmetric coupling effects between the agents.
Although the solution it does find is a local GNE, it is one which is disadvantageous for the lead agent and results in the lead agent being overtaken by the rear agent in both modes.
iLQGames is able to converge to the advantageous local GNE for the first mode and successfully finish the race ahead of the rear agent. However, the nature of the method means that it is only able to keep track of one of the two modes, resulting in the method failing to block the rear agent when the rear agent overtakes via the bottom.
Finally, the multimodal nature of MMELQGames means that it is able to keep track of both local Nash equilibria at the same time.
By using the history of the rear agent's control, the lead agent is able to infer whether rear agent is trying to overtake from the top or the bottom and then execute the corresponding control.

We also note that since MELQGames is a local method, it is prone to becoming stuck in local minima. From \cref{fig:sim:me_wrong}, even though the rear agent is clearly trying to overtake via the bottom mode, the warmstart for iLQGames means that it still predicts the rear agent will overtake from the top.
On the other hand, the multimodality of MMELQGames means that it can keep track of multiple hypothesis at the same time and hence is able to infer the correct mode using the algorithm described in \cref{subsec:bayesian_inference} and respond accordingly.

\section{Conclusion} 
\label{sec:conclusion}
In this paper, we have proposed a constrained \ac{MaxEnt} dynamic game formulation and presented an algorithm that can solve for multiple modes of the corresponding \ac{GNE}.
We demonstrate its advantages over existing unimodal algorithms in the examples of multi-agent collision avoidance and autonomous racing. 

While we have explored the use of constraints on the mean control in our formulation, it may be interesting to look at different ways of including constraints in the \ac{MaxEnt} dynamic game formulation such as via chance constraints or constraints that hold almost-surely.


Finally, though our algorithm identifies multimodal behavior by exploring the state space, we do not provide any guarantees for how thorough this exploration is or whether there are additional modes which have not been discovered.
Quantifying this via uncertainty quantification tools such as Gaussian Processes can provide more structured methods of discovering different modalities in multi-agent interactions.

\section*{Acknowledgments}

\bibliographystyle{plainnat}
\bibliography{references}

\ifx\joinedsupplementary\undefined
\else
    \begin{appendices}
    \onecolumn

\crefalias{section}{appendix}

\makeatletter
\let\origsection\section
\renewcommand\section{\@ifstar{\starsection}{\nostarsection}}

\newcommand\nostarsection[1]
{\sectionprelude\origsection{#1}\sectionpostlude}

\newcommand\starsection[1]
{\sectionprelude\origsection*{#1}\sectionpostlude}

\newcommand\sectionprelude{%
}

\newcommand\sectionpostlude{%
  \setcounter{equation}{0}
}
\makeatother


\renewcommand{\theequation}{\thesection.\arabic{equation}}

\newpage
\section{Expression for Partial Derivatives of $Q^i$}
\label{sm:Q_derivs}
Let $Q^i$ denote the terms inside the expectation in Bellman's equation \eqref{eq:dp:raw_bellman}:
\begin{equation}
    Q^i(\vx, \vu) \coloneqq l^i(\vx, \vu) + {V^i}'( f(\vx, \vu) ).
\end{equation}
Suppose that the value function at the next timestep ${V^i}'$ is quadratic with form
\begin{equation}
    {V^i}'(\vx) = (\bar{V}^i)' + (V^i)'_x \vdx + \frac{1}{2} \vdx\T (V^i)'_{xx} \vdx
\end{equation}
Using a quadratic approximation of the costs and
linear approximation of the dynamics \eqref{eq:me:quad_cost_lin_dyn} then yields
\begin{align}
    Q^i(\vx, \vu) &\approx
    l^i(\vxbar, \vubar)
    + \begin{bmatrix}l^i_x \\ l^i_u\end{bmatrix}\T
      \begin{bmatrix}\vdx \\ \vdu\end{bmatrix}
    + \frac{1}{2}
      \begin{bmatrix}\vdx \\ \vdu\end{bmatrix}\T
      \begin{bmatrix}l^i_{xx} & l^i_{xu} \\ l^i_{ux} & l^i_{uu} \end{bmatrix}
      \begin{bmatrix}\vdx \\ \vdu\end{bmatrix} \\
     &\quad + \Big( f(\vxbar, \vubar) + f_x \vdx + f_u \vdu \Big)\T
     \Big( (\bar{V}^i)' + (V^i)'_x \vdx + \frac{1}{2} \vdx\T (V^i)'_{xx} \vdx \Big)
     \Big( f(\vxbar, \vubar) + f_x \vdx + f_u \vdu \Big).
\end{align}
Expanding and collecting terms then results in
\begin{gather}
\begin{aligned}
    Q^i(\vx, \vu) &\approx \bar{V}^i + \delta Q^i, \\
    \bar{V}^i &\coloneqq l^i(\vxbar, \vubar) + (\bar{V}^i)'(f(\vx, \vu)),
\end{aligned} \\
    \delta Q^i \coloneqq
    \begin{bmatrix}Q^i_x \\ Q^i_u\end{bmatrix}\T \begin{bmatrix}\vdx \\ \vdu\end{bmatrix}
    + \frac{1}{2}
      \begin{bmatrix}\vdx \\ \vdu\end{bmatrix}\T
      \begin{bmatrix}Q^i_{xx} & Q^i_{xu} \\ Q^i_{ux} & Q^i_{uu} \end{bmatrix}
      \begin{bmatrix}\vdx \\ \vdu\end{bmatrix}.
\end{gather}
where the partial derivatives of $Q^i$ are as follows:
\begin{align}
    Q^i_x &= l^i_x + f_x\T (V^i)'_x \\
    Q^i_u &= l^i_u + f_u\T (V^i)'_x \\
    Q^i_{xx} &= l^i_{xx} + f_x\T (V^i)'_{xx} f_x \\
    Q^i_{ux} &= l^i_{ux} + f_u\T (V^i)'_{xx} f_x \\
    Q^i_{uu} &= l^i_{uu} + f_u\T (V^i)'_{xx} f_u
\end{align}

\newpage
\section{Proof of \Cref{thm:me:opt_pol}}
\label{sm:proof:thm:me:opt_pol}
We restate the theorem below for convenience.
\begin{customlemma}{1}[Optimal MELQGames Policy]
The optimal policy $\piis$ which solves \eqref{eq:me:quad_bellman} for each agent $i$ has the form
\begin{equation}
    \pi^{i*} =
    {Z^i}^{-1} \exp\Big(-\frac{\alpha}{2} (\delui - \deluis)\T Q_{\ui\ui}^i (\delui - \deluis) \Big)
\end{equation}
where $\vdus$ is the solution to the following system of equations
\begin{equation}
    0 = \hat{Q}_{uu} \vdus + Q_{ux} \vdx + Q_{u} = k + K \vdx
\end{equation}
and the matrix $\hat{Q}_{uu} \in \Rb^{n_\vu \times n_\vu}$ is obtained by vertically stacking
the row vectors $Q^i_{\ui \vu}$ for each $i \in \{ 1, \dots, N \}$:
\begin{equation}
    \hat{Q}_{uu} \coloneqq
    \begin{bmatrix}
    Q^1_{u^1 \vu} \\
    \vdots \\
    Q^N_{u^N \vu}
    \end{bmatrix}
\end{equation}
\end{customlemma}
\begin{proof}
The optimal policy for agent $i$ in \eqref{eq:me:quad_bellman} is the Gibbs distribution
\begin{equation}
    \pi^{i*}(\delui) \propto
    \exp\left( -\frac{1}{\alpha} \left[ 
        \Big( Q_{\ui}^i + Q_{\ui x}^i \vdx + Q_{\ui \unegi }^i \ExP{}{\vdunegi} \Big)\T \delui + \frac{1}{2} \delui\T \Quui \delui
    \right] \right).
\end{equation}
Since the expression inside the exponential function is quadratic in $\delui$
with the negative definite quadratic term $-\frac{1}{2\alpha} \Quui$, the optimal policy
$\pi^{i*}$ is multivariate Gaussian with covariance matrix $\Sigma = \alpha (\Quui)^{-1}$.
To find the mean of $\pi^{i*}$, we complete the square to obtain
\begin{equation} \label{eq:proof:complete_square}
    \Big( Q_{\ui}^i + Q_{\ui x}^i \vdx + Q_{\ui \unegi }^i \ExP{}{\vdunegi} \Big)\T \delui + \frac{1}{2} \delui\T \Quui = \frac{1}{2} (\delui - \deluis)\T \Quui (\delui - \deluis) - \frac{1}{2} \deluis\T \Quui \deluis
\end{equation}
where the mean $\deluis = \ExP{}{\delui}$ is defined as
\begin{equation} \label{eq:proof:deluis}
    \deluis \coloneqq -(\Quui)^{-1} (Q_{\ui}^i + Q_{\ui x}^i \vdx + Q_{\ui \unegi}^i \vdunegis)
\end{equation}
where we have used $\vdunegis \coloneqq \ExP{}{\vdunegi}$ to denote the mean control of the other agents.
If $\vdunegis$ were known (ex. in the IBR case where $\vdunegis$ is held fixed), then $\deluis$ can be solved from 
\eqref{eq:proof:deluis} alone.
However, in this case, $\vdunegis$ is not known. However, by rearranging \eqref{eq:proof:deluis} we obtain
\begin{align}
    0
    &= \Quui \deluis + Q_{\ui}^i + Q_{\ui x}^i \vdx + Q_{\ui \unegi}^i \vdunegis \label{eq:proof:split_vdus_sys_eq} \\
    &= Q_{\ui}^i + Q_{\ui x}^i \vdx + Q^i_{\ui \vu} \vdus \label{eq:proof:vdus_sys_eq_i},
\end{align}
which provides $n_{\ui}$ equations in the unknown variable $\vdus \in \Rb^{n_{\vu}}$.
Hence, by considering \eqref{eq:proof:vdus_sys_eq_i} for all agents, we obtain the following linear equation for $\vdus$:
\begin{equation}
    \hat{Q}_{uu} \vdus + Q_{ux} \vdx + Q_{u} = 0
\end{equation}
with $\hat{Q}_{uu}$ defined as the stacked matrix of $Q_{\ui \vu}$ for all agents:
\begin{equation}
    \hat{Q}_{uu} \coloneqq
    \begin{bmatrix}
    Q^1_{u^1 \vu} \\
    \vdots \\
    Q^N_{u^N \vu}
    \end{bmatrix}
\end{equation}

\end{proof}

\newpage
\section{Proof of \Cref{thm:me:val_fn}}
\label{sm:proof:thm:me:val_fn}
To simplify the proof, we first begin with the following technical lemma.
\begin{lemma} \label{thm:tech:control_equiv}
Suppose that $\vdus$ satisfies \eqref{eq:me:opt_pol_coupled_eq}. Then,
\begin{align}
    \Big( Q_{\unegi}^i + Q_{\unegi \vx}^i \vdx \Big)\T \vdunegis
    + \frac{1}{2} \vdunegis\T Q_{\unegi \unegi}^i \vdunegis
    -\frac{1}{2} \deluis\T \Quui \deluis
    = \Big({Q_{\vu}^i} + Q_{\vu \vx}^i \vdx \Big)\T \vdus + \frac{1}{2} \vdus\T Q_{\vu \vu}^i \vdus
\end{align}
\end{lemma}
\begin{proof}
Since $\delu$ satisfies \eqref{eq:me:opt_pol_coupled_eq},
it must thus satisfy \eqref{eq:proof:split_vdus_sys_eq}.
Hence,
\begin{equation}
    -\Quui \delui = Q_{\ui}^i + Q_{\ui x}^i \vdx + Q_{\ui \unegi}^i \vdunegis
\end{equation}
Splitting $-\frac{1}{2}$ into $\frac{1}{2} - 1$, using the above identity then simplifying:
\begin{align}
&\phantom{{}={}} \Big( Q_{\unegi}^i + Q_{\unegi \vx}^i \vdx \Big)\T \vdunegis
    + \frac{1}{2} \vdunegis\T Q_{\unegi \unegi}^i \vdunegis
    -\frac{1}{2} \deluis\T \Quui \deluis \\
&= \Big( Q_{\unegi}^i + Q_{\unegi \vx}^i \vdx \Big)\T \vdunegis
    + \frac{1}{2} \vdunegis\T Q_{\unegi \unegi}^i \vdunegis
    +\frac{1}{2} \deluis\T \Quui \deluis - \deluis\T \Quui \delui \\
&= \Big( Q_{\unegi}^i + Q_{\unegi \vx}^i \vdx \Big)\T \vdunegis
    + \frac{1}{2} \vdunegis\T Q_{\unegi \unegi}^i \vdunegis
    +\frac{1}{2} \deluis\T \Quui \deluis
    + \deluis\T (Q_{\ui}^i + Q_{\ui x}^i \vdx + Q_{\ui \unegi}^i \vdunegis) \\
&= \Big({Q_{\vu}^i} + Q_{\vu \vx}^i \vdx \Big)\T \vdus + \frac{1}{2} \vdus\T Q_{\vu \vu}^i \vdus
\end{align}
\end{proof}

We now restate \cref{thm:me:val_fn} below for convenience.
\begin{customlemma}{2}[MELQGames Value Function Update]
Suppose the infimum in the Bellman equation \eqref{eq:dp:raw_bellman}
is solved with policy $\pi^{i*}$ with mean $\vdus = \vk + \vK \vdx$ according
to \cref{thm:me:opt_pol}.
Then, the value function using $\vtildedu$ for agent $i$ has the form
\begin{align} 
    V^i(\vx)
    &= \ExP{\vdu \sim \tilde{\pi}}{
        l^i(\vx, \vu) + {V^i}'( f(\vx, \vu) 
    } - \alpha H[\tilde{\pi}^i] \label{eq:proof:bellman_thing} \\
    &=\Big( V^i + V_H^i \Big) + {V_x^i}\T \vdx + \frac{1}{2} \vdx\T V_{xx}^i \vdx,
\end{align}
where the terms $V^i$, $V_{H}^i$, $V_x^i$ and $V_{xx}^i$ have the form
\begin{align}
V^i
    &= \bar{V}^i + Q_u^i \vk + \frac{1}{2} \vk\T Q_{uu}^i \vk \\
\begin{split}
    V_H^i &= 
        \frac{\alpha}{2} \Big( \log \abs{ Q_{\ui\ui}^i } - n_u \log (2\pi \alpha) \Big) \\
        &\quad + \alpha \sum_{j=1, j \not= i}^N
            \tr\left[ (Q^j_{u^j, u^j})^{-1} Q^i_{u^j, u^j} \right]
\end{split} \\
V_x^i
    &= Q_x^i + \vK\T Q_{uu}^i \vk + \vK\T Q_u^i + Q_{xu}^i \vk \\
V_{xx}^i
    &= Q_{xx}^i + \vK\T Q_{uu}^i \vK + \vK\T Q_{ux}^i + Q_{xu}^i \vK 
\end{align}
\end{customlemma}
\begin{proof}
Let $\mathcal{I}^i$ denote the following integral.
\begin{equation}
\mathcal{I}^i 
\coloneqq \int \exp\left(-\frac{1}{\alpha}\left[
    Q_{\ui}^i + Q_{\ui x}^i \vdx
    + Q_{\ui \unegi }^i \ExP{}{\vdunegi} \Big)\T \delui \\
    + \frac{1}{2}\delui\T Q_{u^i u^i}^i \delui
\right]\right) \ddelui
\end{equation}
Then, completing the square as in \eqref{eq:proof:complete_square} and simplifying, we obtain
\begin{align}
\mathcal{I}^i
&= \exp\left(-\frac{1}{\alpha} \left[-\frac{1}{2}\deluis\T \Quui \deluis \right]\right)
    \int \exp\left(-\frac{1}{\alpha} \left[
        \frac{1}{2} (\delui - \deluis)\T \Quui (\delui - \deluis)
    \right]\right) \ddelui \\
&=\exp\left(-\frac{1}{\alpha} \left[-\frac{1}{2}\deluis\T \Quui \deluis \right]\right)
    \Big( 2 \pi \Big)^{\frac{n_u}{2}} \, \Big( \abs*{ \alpha (\Quui)^{-1}) } \Big)^{\frac{1}{2}} \\
&= \exp\left(-\frac{1}{\alpha} \left[-\frac{1}{2}\deluis\T \Quui \deluis \right]\right)
    \frac{ \Big( 2 \pi \alpha \Big)^{\frac{n_u}{2}} }{ \abs*{ \Quui }^{\frac{1}{2}} } \\
&= \exp\left(-\frac{1}{\alpha} \left[
    -\frac{1}{2}\deluis\T \Quui \deluis
    + \frac{\alpha}{2} \Big( \ln \abs{\Quui} - n_{\ui} \ln( 2 \pi \alpha) \Big)
\right]\right)
\end{align}
Then, the partition function $Z^i$ \eqref{eq:dp:Z_def} takes the form
\begin{align}
Z^i
&= \int \exp\Big( -\frac{1}{\alpha} \ExP{\vunegi}{ {V^i}'(f(\vx,\vu)) + l^i(\vx, \vu) } \Big) \dui \\
\begin{split}
&= \exp\bigg(-\frac{1}{\alpha}\Big[
    \bar{V}^i
    + {Q^i_{\vx}}\T \vdx + \frac{1}{2} \vdx\T Q^i_{\vx \vx} \vdx
    + \Big( Q^i_{\vunegi} + Q^i_{\vunegi \vx} \vdx \Big)\T \vdunegis
    + \frac{1}{2} \vdunegis\T Q_{\unegi \unegi}^i \vdunegis \\
&\qquad\qquad + \alpha \sum_{j=1, j \not= i}^N
            \tr\left[ (Q^j_{u^j, u^j})^{-1} Q^i_{u^j, u^j} \right]
\Big]\bigg) \\
&\quad \int \exp\Big( -\frac{1}{\alpha} \left[
     \Big( Q_{\ui}^i + Q_{\ui x}^i \vdx
    + Q_{\ui \unegi }^i \ExP{}{\vdunegi} \Big)\T \delui
    + \frac{1}{2}\delui\T Q_{u^i u^i}^i \delui
\right]\Big) \ddelui \\
\end{split} \\
\begin{split}
&= \exp\bigg(-\frac{1}{\alpha}\Big[
    \bar{V}^i
    + {Q^i_{\vx}}\T \vdx + \frac{1}{2} \vdx\T Q^i_{\vx \vx} \vdx
    + \Big( Q^i_{\vunegi} + Q^i_{\vunegi \vx} \vdx \Big)\T \vdunegis
    + \frac{1}{2} \vdunegis\T Q_{\unegi \unegi}^i \vdunegis \\
&\qquad\qquad + \alpha \sum_{j=1, j \not= i}^N
            \tr\left[ (Q^j_{u^j, u^j})^{-1} Q^i_{u^j, u^j} \right]
\Big]\bigg) \; \mathcal{I}
\end{split} \\
&= \exp\left(-\frac{1}{\alpha} V^i(\vx) \right)
\end{align}
where the trace term comes from taking the expectation of the quadratic term in $\vdunegi$ and
\begin{equation}
\begin{split}
V^i(\vx) &=
(\bar{V}^i)'
    + {Q^i_{\vx}}\T \vdx + \frac{1}{2} \vdx\T Q^i_{\vx \vx} \vdx
    + \Big( Q^i_{\vunegi} + Q^i_{\vunegi \vx} \vdx \Big)\T \vdunegis
    + \frac{1}{2} \vdunegis\T Q_{\unegi \unegi}^i \vdunegis \\
&\qquad -\frac{1}{2}\deluis\T \Quui \deluis
    + \frac{\alpha}{2} \Big( \ln \abs{\Quui} - n_{\ui} \ln( 2 \pi \alpha) \Big)
    + \alpha \sum_{j=1, j \not= i}^N
            \tr\left[ (Q^j_{u^j, u^j})^{-1} Q^i_{u^j, u^j} \right]
\end{split}
\end{equation}
Applying \cref{thm:tech:control_equiv}, using the definition of $\vdus$ and collecting terms then yields
\begin{align}
\begin{split}
V^i(\vx)
&=  (\bar{V}^i)'
    + {Q^i_{\vx}}\T \vdx + \frac{1}{2} \vdx\T Q^i_{\vx \vx} \vdx
    + \Big({Q_{\vu}^i} + Q_{\vu \vx}^i \vdx \Big)\T \vdus \\
&\qquad + \frac{1}{2} \vdus\T Q_{\vu \vu}^i \vdus
    + \frac{\alpha}{2} \Big( \ln \abs{\Quui} - n_{\ui} \ln( 2 \pi \alpha) \Big)
    + \alpha \sum_{j=1, j \not= i}^N
            \tr\left[ (Q^j_{u^j, u^j})^{-1} Q^i_{u^j, u^j} \right]
\end{split} \\
&= \Big( V^i + V_H^i \Big) + {V_x^i}\T \vdx + \frac{1}{2} \vdx\T V_{xx}^i \vdx
\end{align}
where
\begin{align}
V^i
    &= \bar{V}^i + Q_u^i \vk + \frac{1}{2} \vk\T Q_{uu}^i \vk \\
\begin{split}
    V_H^i &= 
        \frac{\alpha}{2} \Big( \log \abs{ Q_{\ui\ui}^i } - n_u \log (2\pi \alpha) \Big) \\
        &\quad + \alpha \sum_{j=1, j \not= i}^N
            \tr\left[ (Q^j_{u^j, u^j})^{-1} Q^i_{u^j, u^j} \right]
\end{split} \\
V_x^i
    &= Q_x^i + \vK\T Q_{uu}^i \vk + \vK\T Q_u^i + Q_{xu}^i \vk \\
V_{xx}^i
    &= Q_{xx}^i + \vK\T Q_{uu}^i \vK + \vK\T Q_{ux}^i + Q_{xu}^i \vK 
\end{align}
\end{proof}

\newpage
\section{Proof of \Cref{lemma:aug_lang_cvg}}
\label{sm:proof:lemma:aug_lang_cvg}
We restate \cref{lemma:aug_lang_cvg} below for convenience.
\begin{customlemma}{3}
Suppose that the augmented Lagrangian MELQGames converges to a tuple $(\pi^*, \lambda^i_j, \rho)$ for some $\rho > 0$ which satisfies dual feasibility of $\lambda$, \ac{LICQ} and \ac{SOSC}.
Then, $\pi^*$ is a \ac{GNE} for \eqref{eq:pf:gnep}.
\end{customlemma}
\begin{proof}
Since the solution satisfies \ac{LICQ}, \ac{SOSC} and $\lambda > 0$,
each $\piis$ is a local minima for its corresponding constrained optimal control problem \eqref{eq:pf:gnep}.
Thus, there does not exist any feasible control $\pi \in \Pi(\pinegis)$ within a neighborhood of $\piis$ that can achieve a lower cost.
Hence, $\pi^*$ satisfies \eqref{eq:pf:gne_nec} and is a \ac{GNE}.
\end{proof}

\newpage
\section{Proof of \Cref{thm:mme:common_belief}}
\label{sm:sec:proof:mme:common_belief}
Before proving \cref{thm:mme:common_belief}, we first introduce the following theorem from
\citet{maschler2013game} which presents a sufficient condition for an event $A$ to be common belief
among the players at a particular state of the world $\omega$:

\begin{theorem} \label{sm:thm:suff_cond_common_belief}
Let $N$ denote the number of agents in the game,
$\omega \in Y$ be a state of the world and $A \subseteq Y$ be an event satisfying the following two conditions:
\begin{align}
    p_i( A | \omega ) &= 1, \quad \forall i \in \{ 1, \dots, N \} \\
    p_i( A | \omega' ) &= 1, \quad \forall i \in \{ 1, \dots, N \},
    \; \forall \omega' \in A
\end{align}
Then, $A$ is common belief among the agents at $\omega$.
\end{theorem}

Using \cref{sm:thm:suff_cond_common_belief}, we can now prove 
\eqref{thm:mme:common_belief}, which we restate below for convenience.
\begin{customlemma}{4}
    Let $E = \{ \omega_{a,s} \}_{a=1}^A$ denote the event that the ego agent knows the correct mode with probability 1.
    Then, for all $i \not= 1$, $\omega \in B_i E$, i.e. all non-ego agents believe that the ego agent knows the mode perfectly.
    Furthermore, it is common belief that $B_i E_a$ for all $i \not = 1$.
    In particular, the ego agent (correctly) believes that all non-ego agents (incorrectly) believes that it knows the correct mode.
\end{customlemma}
\begin{proof}
First, note that at $\omega \in E$, all agents, including the ego agent, know the true mode.
Using the definition of the belief operator $B_i$ \eqref{eq:def:belief},
we have that for $i \not = 1$:
\begin{equation}
    B_i E = \{ \omega \in Y | p_i(E|\omega) = 1\} = Y.
\end{equation}
In other words, non-ego agent $i$ believes that $E$ will attains no matter what the true world state $\omega_*$ is.
Since $p_i(Y|\omega) = 1$ for any $\omega \in Y$, $B_i E = Y$ is thus common belief among the agents at all world states.
Hence, by the definition of common belief, we have that
\begin{equation}
  B_1 B_i E = B_1 Y = Y  
\end{equation}
i.e., the ego agent believes that the non-ego agents believe it knows the true mode.
On the other hand, we have that
\begin{align}
    B_1 E &= \{ \omega \in Y | p_1(E | \omega) = 1\} = E, \\
    B_1 E^\complement &= \{ \omega \in Y | p_1(E^\complement|\omega) = 1\}
    = E^\complement.
\end{align}
where $E^\complement$ denotes the complement of $E$ in $Y$.
In other words, when $\omega \in E^\complement$, the ego agent (correctly) believes
that it does not know the true mode while simultaneously knowing that the non-ego agents (incorrectly) believe that it knows the true mode.
\end{proof}

    \end{appendices}
\fi

\end{document}